\documentclass[12pt]{amsart}
\usepackage[english]{babel}
\usepackage{graphicx}
\usepackage{amsmath,amssymb,hyperref,srcltx}
\newcommand{\diff}[2]{\mbox{{\rm Diff}{${\,}_{#1}({\mathbb C}^{#2},0)$}}}
\newcommand{\diffh}[2]{\mbox{$\widehat{\rm Diff}{{\,}_{#1}({\mathbb C}^{#2},0)}$}}

\newcommand{\cn}[1]{\mbox{(${\mathbb C}^{#1},0$)}}

\newcommand{\convn}{{\mathbb C} \{ x, x_{1},\ldots,x_{n} \}}
\newcommand{\formn}{{\mathbb C}[[x, x_{1},\ldots,x_{n}]]}

\newcommand{\ox}{\'{o}}

\newtheorem{pro}{Proposition}[section]
\newtheorem{teo}{Theorem}[section]
\newtheorem*{main}{Main Theorem}

\newtheorem{lem}{Lemma}[section]
\newtheorem{rem}{Remark}[section]
\newtheorem{defi}{Definition}[section]
\newtheorem{que}{Question}
\begin{document}

\title[Extension of formal conjugations]
{Extension of formal conjugations between diffeomorphisms}

\author{Javier Rib\ox n}
\address{Instituto de Matem\'{a}tica, UFF, Rua M\'{a}rio Santos Braga S/N
Valonguinho, Niter\'{o}i, Rio de Janeiro, Brasil 24020-140}
\thanks{e-mail address: javier@mat.uff.br}
\thanks{MSC-class. Primary: 37F45; Secondary: 37G10, 37F75, 34E05, 30E15}
\date{\today}
\maketitle

\bibliographystyle{alpha}
\section*{Abstract}
We study the formal conjugacy properties of germs of complex analytic
diffeomorphisms
defined in the neighborhood of the origin of ${\mathbb C}^{n}$.
More precisely, we are interested on the nature of formal conjugations
along the fixed points set. We prove that there are formally conjugated
local diffeomorphisms $\varphi, \eta$ such that every formal conjugation
$\hat{\sigma}$ (i.e. $\eta \circ \hat{\sigma} = \hat{\sigma} \circ \varphi$)
does not extend to the fixed points set  $Fix (\varphi)$ of $\varphi$,
meaning that it is not  transversally formal (or semi-convergent) along
$Fix (\varphi)$.

We focus on unfoldings of $1$-dimensional tangent to the identity
diffeomorphisms.  We identify the geometrical configurations preventing
formal conjugations to extend to the fixed points set: roughly speaking,
either the unperturbed fiber is singular or generic fibers contain multiple
fixed points.

Keywords: resonant diffeomorphism, bifurcation theory,
asymptotic expansions, formal classification, potential theory.
\section{Introduction}
The study of normal forms and normalizing applications is a classical topic
in dynamical systems (see the introduction of \cite{PM1}).
In this spirit we are interested on studying the nature of formal conjugations
between local complex analytic diffeomorphisms.

Divergent power series are associated to analytic dynamical systems in a natural
way. For instance they can appear when calculating first integrals, linearizing maps, etc.
by using the method of undetermined coefficients.
Cauchy already noticed that the divergent series could be useful even if
their use was devoided of any rigor.
Astronomers dealt with these series by summing them up to the smallest term, obtaining
good approximations to solutions. The error of the approximations
is of the order of the smallest term in many practical examples
(see \cite{pham} for a much more detailed explanation).
Poincar\'{e} and Borel studied why these methods are successful.
The theory experienced a boost in the seventies and eighties when the
concept of summable power series and resurgent functions were introduced
by  Ecalle  \cite{Ecalle:resurg1}.
Roughly speaking summable and multi-summable power series solutions represent
analytic solutions defined in
sectorial domains that do not match to provide a unique analytic solution.
Ramis proved that the solutions of analytic linear O.D.E. are multisummable
(see \cite{MalRam:Fou} and \cite{BBRS}). This result was generalized
by Braaksma for non-linear analytic O.D.E. \cite{Braaksma}.
The objects of the previous results are always one variable analytic functions.
The theory of divergent series in several variables is much more difficult
and it is still evolving. The goal of this paper is introducing and studying new properties
of formal conjugacies between several variable complex analytic diffeomorphisms.

  Determining whether two complex analytic germs of diffeomorphism at $\cn{n}$ are formally
conjugated is a classical problem. It is a first step towards an analytic classification
but the formal classification has an intrinsic interest since formal conjugations can support
an underlying geometric nature.
This geometrical structure is typically revealed by showing the summability or
resurgence of the formal objects.

For instance consider
$\varphi_{1}=z + {z}^{p+1} + h.o.t.$ a germ of
diffeomorphism at $\cn{}$.
It is the exponential of a formal vector field (i.e. a derivation of the
ring of formal power series) of the form
$(z^{p+1} + h.o.t.) \partial / \partial z$. It is the so called infinitesimal generator of
$\varphi_{1}$ that we denote $\log \varphi_{1}$.
Its expression is obtained by using undetermined coefficients in the Taylor's formula
\begin{equation}
\label{equ:taylor1}
 {\rm exp}(X) = {\rm exp} (a(z) \partial / \partial z) =
z + \sum_{j=1}^{\infty} \frac{{X}^{j}(z)}{j!} .
\end{equation}
where $X(z)=a(z)$ and $X^{j+1}(z) = X(X^{j}(z))$ for any $j \geq 1$.
For instance we have $X^{2}(z)= a(z) a'(z)$.
There exist sectors $V_{1}$, $\hdots$, $V_{2p}$ whose union is a pointed neighborhood of
$0$ where the infinitesimal generator of $\varphi_{1}$ provides
analytic flows in which the discrete dynamics of $\varphi_{1}$ is embedded.
More precisely the infinitesimal generator of $\varphi_{1}$ is
the $p$-Gevrey asymptotic development
of a unique analytic vector field
$X_{j} = a_{j}(z) \partial / \partial z$ defined in $V_{j}$
such that $\varphi_{1} = {\rm exp}(X_{j})$ for any $1 \leq j \leq 2p$.
Moreover Martinet and Ramis showed that $\log \varphi_{1}$ is p-summable \cite{MaRa:ihes}.
The infinitesimal generator of $\varphi_{1}$ generically diverges
since the vector fields $X_{j}$ do not coincide in the intersection of
their domains of definition.

The situation for the formal conjugacy problem is analogous.
Consider a germ of diffeomorphism $\varphi_{2}$ formally conjugated to
$\varphi_{1}$. Again
every formal conjugation $\hat{\sigma}$ is the $p$-Gevrey asymptotic development
of a unique $\sigma_{j} \in {\mathcal O}(V_{j})$ satisfying
$\sigma_{j} \circ \varphi_{1} = \varphi_{2} \circ \sigma_{j}$.
Moreover $\hat{\sigma}$ is p-summable \cite{MaRa:ihes}.
In general there is no convergent choice of $\hat{\sigma}$.
In the particular case where  $\varphi_{1}$ and $\varphi_{2}$ are
embedded in analytic flows then $\hat{\sigma}$
is convergent.
But the infinitesimal generators are in general only $p$-summable and this
forces the conjugation $\hat{\sigma}$ to be also $p$-summable.
In this example formal conjugations are as badly behaved as the infinitesimal
generators. One of the goals of this paper is proving that conjugations
can be more pathological than infinitesimal generators.
\subsection{Properties of formal conjugacies along fixed points}
We denote ${\rm Diff} ({\mathbb C}^{n},q)$ the group of complex analytic germs of
diffeomorphism at $q=(q_{1},\hdots,q_{n}) \in {\mathbb C}^{n}$. Let $\widehat{\rm Diff} ({\mathbb C}^{n},q)$
be the formal completion of ${\rm Diff} ({\mathbb C}^{n},q)$
with respect to the filtration
$\{{\mathfrak m}^{k} \times \hdots \times {\mathfrak m}^{k} \}_{k \in {\mathbb N} \cup \{0\}}$
where ${\mathfrak m}$ is the maximal ideal of ${\mathbb C}\{y_{1}-q_{1},\hdots,y_{n}-q_{n}\}$
(see \cite{Eisenbud}, section 7.1). The composition in
$\widehat{\rm Diff} ({\mathbb C}^{n},q)$ is defined in the natural way by taking the composition in
${\rm Diff} ({\mathbb C}^{n},q)$ and passing to the limit in the Krull topology
(see \cite{Eisenbud}, page 204).

We focus
on the action of  formal conjugations on the fixed points sets of
diffeomorphisms.
Consider coordinates $(y_{1},\hdots,y_{n}) \in {\mathbb C}^{n}$.
We say that $\hat{\sigma} \in \diffh{}{n}$ is {\it transversally formal} (t.f. for shortness)
along a germ of analytic set $\gamma$ given by an ideal $I(\gamma)$ if
$y_{j} \circ \hat{\sigma}$ can be expressed in the form
\[ y_{j} \circ \hat{\sigma} = \sum_{k=0}^{\infty} c_{j,k}  \]
where $c_{j,k} \in {\mathbb C} \{y_{1},\hdots,y_{n} \} \cap I(\gamma)^{k}$
for all $k \geq 0$ and $1 \leq j \leq n$.
Formal transversality is also called semi-convergence sometimes since
intuitively the formal series is convergent in the direction of $\gamma$ and can diverge
in the direction transversal to $\gamma$.
If we can choose $c_{j,k} \in {\mathcal O}(U)$ for
all $k \geq 0$, $1 \leq j \leq n$ and some neighborhood $U$ of
$0$ we say that $\hat{\sigma}$ is uniformly transversally formal
(u.t.f.) along $\gamma$. Denote by $Fix (\sigma)$ the fixed points set
of $\sigma \in \diff{}{n}$. In this paper we prove:
\begin{main}
\label{teo:intr1}
Let $n \geq 2$. There exist $\varphi_{1},\varphi_{2} \in \diff{}{n}$ such that
$\hat{\sigma} \circ \varphi_{1} = \varphi_{2} \circ \hat{\sigma}$ for some $\hat{\sigma} \in \diffh{}{n}$
but no choice of $\hat{\sigma}$ is t.f. along $Fix (\varphi_{1})$.
\end{main}
Next, we  explain the interest of this result.
Given $\varphi \in \diff{}{n}$ and a point $q \in Fix (\varphi)$ close to $0$
we define the germ $\varphi_{q} \in {\rm Diff} ({\mathbb C}^{n},q)$ obtained
by restricting $\varphi$ to a neighborhood of $q$. Then we
can consider $\varphi \in \diff{}{n}$ as a family ${(\varphi_{q})}_{q \in U \cap Fix (\varphi)}$
for some neighborhood $U$ of $0$. We can interpret $\varphi$ as a germ defined in
the neighborhood of any of its fixed points.
This situation provides the motivation to consider the following questions:
\begin{que}
\label{que:first}
Suppose that $\varphi_{1}, \varphi_{2} \in \diff{}{n}$ are formally conjugated.
Does there exist a formal conjugation whose action on $Fix (\varphi_{1})$ is convergent?
\end{que}
\begin{que}
\label{que:second}
Suppose that $\varphi_{1}, \varphi_{2} \in \diff{}{n}$ are formally conjugated by
$\hat{\sigma} \in \diffh{}{n}$. Assume that the action of $\hat{\sigma}$ on $Fix (\varphi_{1})$ is
convergent. Are the germs $\varphi_{1,q}$ and $\varphi_{2,\hat{\sigma}(q)}$ formally conjugated
for any $q \in U \cap Fix (\varphi_{1})$ and some neighborhood $U$ of $0$?
\end{que}
\begin{que}
\label{que:third}
Suppose that the answer of question (\ref{que:second}) is affirmative for
$\varphi_{1}$, $\varphi_{2}$ and $\hat{\sigma}$.
Does there exist $\hat{\tau} \in \diffh{}{n}$ such that $\hat{\tau}$ acts on $Fix (\varphi_{1})$
as $\hat{\sigma}$ and induces a formal conjugation between
$\varphi_{1,q}$ and $\varphi_{2,\hat{\sigma}(q)}$ for any $q \in  V \cap Fix (\varphi_{1})$
and some neighborhood $V$ of $0$?
\end{que}
We can replace the fixed points with the periodic points of order $p$ for every $p \in {\mathbb N}$
in the previous questions.

We are interested on studying the nature of formal conjugations and its
geometrical meaning.
Naturally we focus initially on the fixed points set.
Of course we have to consider non-generic diffeomorphisms since otherwise
$Fix (\varphi) =\{ 0 \}$ and the answer is obviously
affirmative for the three questions. In this paper we work with
unfoldings $(h(x,x_{1},\hdots,x_{n}),x_{1},\hdots,x_{n})$.
This is an interesting case in itself.
We want to understand
how the structure of a formal conjugation extends
to the fixed points set for the perturbed diffeomorphisms.

Let us clarify the previous statements.
  We say that $\hat{\sigma} \in \diffh{}{n}$ has a convergent action on an analytic germ of analytic
set $\gamma$ if there exists $\sigma \in \diff{}{n}$ such that
\[ \hat{\sigma} \circ {\sigma}^{(-1)} - Id \in I(\gamma) \times \hdots \times I(\gamma) \]
where $I(\gamma)$ is the ideal of $\gamma$. We denote $\hat{\sigma}(q) = \sigma(q)$ for any $q \in \gamma$.

For the questions (\ref{que:second}) and (\ref{que:third}) we can suppose that
\[ \hat{\sigma} - Id \in I(Fix (\varphi_{1})) \times \hdots \times I(Fix (\varphi_{1})) \]
by replacing $\hat{\sigma}$ with $\hat{\sigma} \circ \sigma^{(-1)}$. In question  (\ref{que:third})
we want to know whether $\hat{\sigma}$ can be extended to an irreducible component $\gamma$
of $Fix (\varphi_{1})$. The strong form of question (\ref{que:third}) asks whether or not
$\hat{\sigma}$ is u.t.f. along $Fix (\varphi_{1})$. In such a case
$\hat{\sigma}$ belongs to $\widehat{\rm Diff} ({\mathbb C}^{n},q)$ and conjugates
$\varphi_{1,q}$ and $\varphi_{2,q}$ for any $q \in U \cap Fix (\varphi_{1})$ and some
neighborhood $U$ of $0$. In the weak version of question (\ref{que:third}) we wonder
whether $\hat{\sigma}$ is t.f. along $Fix (\varphi_{1})$.
\subsection{Analyzing the questions}
We answer the questions for germs of finite-dimensional unfoldings of elements of
$\diff{}{}$.  Consider coordinates $(x,x_{1},\hdots,x_{n}) \in {\mathbb C}^{n+1}$.
We define the group
\[ \diff{p}{n+1} = \{ \varphi \in \diff{}{n+1} : x_{j} \circ \varphi = x_{j} \
{\rm for \ any} \ 1 \leq j \leq n \}  \]
where $n \in {\mathbb N} \cup \{0 \}$. We define
\[ \diff{up}{n+1} = \{ \varphi \in \diff{p}{n+1} : (\partial{(x \circ \varphi)}/\partial{x})(0)=1 \}  . \]
It is the group of $n$-parameter unfoldings of tangent to the identity diffeomorphisms.
Denote by ${\rm Diff}_{r} ({\mathbb C}^{n+1},0)$ the subgroup
of $\diff{p}{n+1}$ whose elements satisfy
$\varphi^{(k)} \in \diff{up}{n+1}$ for some $k \in {\mathbb N}$.
Then ${\rm Diff}_{r} ({\mathbb C}^{n+1},0)$ is the group of
$n$-parameter unfoldings of resonant diffeomorphisms.

The answer to the questions is affirmative if
$\varphi_{1} \not \in  {\rm Diff}_{r} ({\mathbb C}^{n+1},0)$.
Indeed $\varphi_{1}$ is t.f. conjugated to
$(a(x_{1},\hdots,x_{n}) x, x_{1},\hdots,x_{n})$ for some unit
$a \in {\mathbb C}\{x_{1},\hdots,x_{n}\}$.
It suffices to work in ${\rm Diff}_{r} ({\mathbb C}^{n+1},0)$.
The next proposition implies that we can reduce the study to $\diff{up}{n+1} $.
\begin{pro}
Let $\varphi_{1},\varphi_{2} \in {\rm Diff}_{r} ({\mathbb C}^{n+1},0)$. Then
$\varphi_{1}$, $\varphi_{2}$ are formally conjugated if and only if
$(\partial (x \circ \varphi_{1})/\partial x)(0) =
(\partial (x \circ \varphi_{2})/\partial x)(0)$
and
$\varphi_{1}^{k}, \varphi_{2}^{k}$
are formally conjugated ($k$ is the period of
$(\partial (x \circ \varphi_{1})/\partial x)(0)$).
\end{pro}
A stronger result can be proved. Suppose that $\hat{\sigma} \circ
\varphi_{1}^{k} =  \varphi_{2}^{k} \circ \hat{\sigma}$ for some
$\hat{\sigma} \in \diffh{}{n}$ and $(\partial (x \circ
\varphi_{1})/\partial x)(0) = (\partial (x \circ
\varphi_{2})/\partial x)(0)$. Then either $\varphi_{1}^{k} \equiv Id
\equiv \varphi_{2}^{k}$ and $\varphi_{1}$, $\varphi_{2}$ are
analytically conjugated or $\hat{\sigma} \circ \varphi_{1} =
\varphi_{2} \circ \hat{\sigma}$. These results are contained in
prop. 5.4 in \cite{JR:mod}. There the context is a bit different but
the generalization is straightforward.

The answer to question  (\ref{que:first}) is affirmative if
$\varphi_{1} \in \diff{up}{n+1}$ (proposition \ref{pro:action}).

To answer question  (\ref{que:second}) in $\diff{up}{n+1}$ we have
to divide the irreducible components of $Fix (\varphi)$ in two sets.
Let $\gamma$ be an irreducible hypersurface of $Fix(\varphi)$ for some
 $p \in {\mathbb N}$.
We say that $\gamma$ is {\it unipotent} with respect to $\varphi$ if
$( {\partial{(x \circ \varphi)} / \partial{x} )}_{|\gamma} \equiv 1$.
If $\gamma$ is unipotent with respect to $\varphi_{1}$
then the answer to question  (\ref{que:second}) is positive for any
$q \in \gamma$ (prop. \ref{pro:prifp}). If $\gamma$ is non-unipotent
the result still holds true for any
$q \in \gamma$ such that $|(\partial{(x \circ \varphi_{1})}/\partial{x})(q)| \neq 1$.
We can not extend the result to every point of $\gamma$
(remark \ref{rem:nutf}). Anyway the property in question  (\ref{que:second})  is satisfied for
generic fixed points.

Let us notice that the answers of questions (\ref{que:first}) and (\ref{que:second})
are consequences of the results in \cite{UPD}. This paper is intended to deal
with question (\ref{que:third}).
\section{Studying formal transversality}
Let us address question (\ref{que:third}).
By the previous discussion we can not expect a formal conjugation $\hat{\sigma}$
to be u.t.f. along a
non-unipotent hypersurface of $Fix (\varphi_{1})$. But we could hope for the
formal conjugations to
be t.f. along $Fix (\varphi_{1})$ and u.t.f. along the unipotent irreducible components of
$Fix (\varphi_{1})$.

  Fix $f \in \convn$ such that $f(0)=(\partial{f}/\partial{x})(0)=0$. The function
$x \circ \varphi -x$ is of the previous form for any $\varphi \in \diff{up}{n+1}$. We define
\[ {\mathcal D}_{f} = \{ \varphi \in \diff{up}{n+1} : (x \circ \varphi - x)/f \ {\rm is \ a \ unit} \} . \]
It is the set of unfoldings whose fixed points set is $f=0$.
Fix $\gamma$ an irreducible component of $f=0$. We have
$(\partial{(x \circ \varphi)} / \partial{x})_{|\gamma} \equiv 1$ if and only if
$(\partial f / \partial x)_{|\gamma} \equiv 0$
for any $\varphi \in {\mathcal D}_{f}$.
We say that an irreducible component $\gamma$ of $f=0$ is {\it unipotent}
if $(\partial f/\partial x)_{|\gamma} \equiv 0$.

We say that a germ of variety is {\it fibered} if it is a union of orbits of
$\partial/\partial{x}$.
By definition $\hat{\sigma} \in \diffh{p}{n+1}$ is {\it normalized} with respect to $f=0$
if $x \circ \hat{\sigma} - x \in I(\gamma)$ for any non-fibered irreducible component
$\gamma$ of $f=0$. The points in non-fibered components of
$f=0$ are fixed points of a normalized transformation $\hat{\sigma}$.

Let $\varphi_{1},\varphi_{2} \in {\mathcal D}_{f}$; we denote $\varphi_{1} {\sim}_{*} \varphi_{2}$
if $\varphi_{1}$ and $\varphi_{2}$ are conjugated by a normalized
$\hat{\sigma} \in \diffh{p}{n+1}$.
If $\hat{\sigma}$ can be chosen t.f. along $f=0$ then we we denote
$\varphi_{1} \stackrel{t}{\sim}_{*} \varphi_{2}$.
Finally if we can choose $\hat{\sigma}$ to be t.f. along $f=0$ and u.t.f. along the
irreducible unipotent components
of $f=0$ we denote $\varphi_{1} \stackrel{ut}{\sim}_{*} \varphi_{2}$.
\begin{teo}
\label{teo:intr2}
  Let $f=x^{a}(x-x_{1})^{b}$ for some $(a,b) \in {\mathbb N} \times {\mathbb N}$.
Then there exists $\varphi_{1},\varphi_{2} \in {\mathcal D}_{f} \subset \diff{up}{2}$
such that $\varphi_{1} {\sim}_{*} \varphi_{2}$ but
$\varphi_{1} \not \stackrel{t}{\sim}_{*} \varphi_{2}$.
\end{teo}
\begin{teo}
\label{teo:intr3}
  Let $f=(x_{2}-xx_{1})^{c}$ for some $c \in {\mathbb N}$.
Then there exists $\varphi_{1},\varphi_{2} \in {\mathcal D}_{f} \subset \diff{up}{3}$
such that $\varphi_{1} {\sim}_{*} \varphi_{2}$ but
$\varphi_{1} \not \stackrel{t}{\sim}_{*} \varphi_{2}$.
\end{teo}
  Questions (\ref{que:first}) and (\ref{que:second}) are true in $\diff{p}{n+1}$ once we introduce the proper setup.
In spite of this the answer to question (\ref{que:third}) is negative.
This is a corollary of theorems \ref{teo:intr2} and \ref{teo:intr3}.
These theorems also imply the Main Theorem. Moreover formal conjugations are
more pathological than infinitesimal generators since the latter ones
are always t.f. along the fixed points set (prop. \ref{pro:logtra}).

The nature of the examples provided by theorems \ref{teo:intr2} and \ref{teo:intr3} is different.
For the former one we can choose a normalized $\hat{\sigma}_{j} \in \diffh{p}{2}$ conjugating
$\varphi_{1}$ and $\varphi_{2}$ such that $\hat{\sigma}_{j}$ is t.f. along $x=jx_{1}$ for
$j \in \{0,1\}$. Nevertheless $\hat{\sigma}_{0}$ is not t.f. along $x=x_{1}$ whereas
$\hat{\sigma}_{1}$ is not t.f. along $x=0$. The existence of multiple fixed points in
the fibers $x_{1}=cte$ imposes incompatible conditions on a formal conjugacy in order to
be transversally formal.
For the other example the lack of t.f. conjugations
is associated to the bad position of the smooth hypersurface $x_{2}-xx_{1}=0$
with respect to $\partial/\partial{x}$.
Let us remark that no fiber of  $d x_{1}=d x_{2}=0$ contains more than one fixed
point of $x_{2}-xx_{1}=0$ except the unperturbed fiber $x_{1}=x_{2}=0$.
\subsection{Outline of the proofs}
 Fix $f \in \convn$ such that
$f$ and $\partial{f}/\partial{x}$ vanish at $0$.
We can characterize the properties
$\varphi_{1} \stackrel{t}{\sim}_{*} \varphi_{2}$ and $\varphi_{1} \stackrel{ut}{\sim}_{*} \varphi_{2}$
for $\varphi_{1}, \varphi_{2}$ in ${\mathcal D}_{f}$.
The idea is that classes of formal conjugacy are path-connected and that the study of
conjugations and their properties can be reduced to analyze the tangent space of paths
and in particular some ordinary differential
equations.

The infinitesimal generator $\log \varphi_{k}$ of $\varphi_{k}$ is of the form
$\hat{u}_{k} f \partial / \partial x$ where $\hat{u}_{k}$ is a unit of
the ring $\formn$ for $k \in \{1,2\}$. In other words we have
$\varphi_{k} = {\rm exp}(\hat{u}_{k} f \partial / \partial x)$
(see equations (\ref{equ:taylor1}) and
(\ref{equ:taylor2})). We associate to
$\varphi_{1}$ and $\varphi_{2}$ the so called {\it homological equation}
\[  \frac{\partial \alpha}{\partial{x}}  =
\left({ \frac{1}{\hat{u}_{1}} - \frac{1}{\hat{u}_{2}} }\right) \frac{1}{f} . \]
We characterize the classes of equivalence of the relations
${\sim}_{*}$, $\stackrel{t}{\sim}_{*}$ and $\stackrel{ut}{\sim}_{*}$
in terms of the homological equation (theorems \ref{teo:UPD} and \ref{teo:cartra}).
In particular if the equation has no poles, i.e. $\hat{u}_{1} - \hat{u}_{2} \in (f)$, then we
obtain $\varphi_{1} {\sim}_{*} \varphi_{2}$ (theorem \ref{teo:UPD}).

Consider $\varphi_{1}, \varphi_{2} \in {\mathcal D}_{f}$ whose homological equation
is free of poles and suppose for simplicity that $f=0$
has no fibered irreducible components in the rest of this section.
Any solution $\hat{\alpha}$ of the homological equation in the field of fractions of
$\formn$ belongs to $\formn$. We say that $\hat{\alpha}$
{\it converges by restriction} to $f=0$ if there exists
$\kappa$ in ${\mathbb C}\{x,x_{1},\hdots,x_{n}\}$
such that $\hat{\alpha} - \kappa \in \sqrt{(f)}$ where $ \sqrt{(f)}$ is the radical of
the ideal $(f)$ of $\formn$.
Then  $\varphi_{1} \stackrel{t}{\sim}_{*} \varphi_{2}$ (resp.
$\varphi_{1} \stackrel{ut}{\sim}_{*} \varphi_{2}$) is equivalent to the existence of
a solution of the homological equation converging by restriction to $f=0$
(see theorem \ref{teo:cartra} for the most general result).

Our goal is proving that there exist solutions of the homological equations that
are not convergent by restriction to the fixed points set.
Then we obtain theorem \ref{teo:intr2}. We define
\[ {\mathcal D}_{f}' =
\{ \varphi \in \diff{up}{n+1} :  x \circ \varphi - x \circ {\rm exp}(f \partial/\partial{x}) \in (f^{2}) \}. \]
The set ${\mathcal D}_{f}'$ is a subset of ${\mathcal D}_{f}$.
The homological equation associated to ${\rm exp}(f \partial/\partial{x}) $
and $\varphi \in {\mathcal D}_{f}'$ has no poles; we deduce
${\rm exp}(f \partial/\partial{x}) {\sim}_{*} \varphi$.
Fix $\varphi_{0} \in {\mathcal D}_{f}'$ and $\Delta$ in $\convn$.
We consider the family ($\lambda \in {\mathbb C}$)
\[ \varphi_{\lambda,\Delta} = (x \circ \varphi_{0} + \lambda f^{2} \Delta, x_{1}, \hdots , x_{n})
\in {\mathcal D}_{f}'  . \]
We introduce the {\it derived equation}
$\partial \alpha / \partial x=
(\partial \hat{K}_{\lambda}/\partial \lambda)_{|\lambda=0}/f$
where $\partial \alpha / \partial x=  \hat{K}_{\lambda}/f$ is the homological equation
associate to ${\rm exp}(f \partial/\partial{x}) $ and $\varphi_{\lambda,\Delta}$.
This equation is key since the property
${\rm exp}(f \partial/\partial{x}) \stackrel{t}{\sim}_{*} \varphi_{\lambda,\Delta}$
for any $\lambda \in {\mathbb C}$ implies that there exists a solution of the derived equation
converging by restriction to $f=0$. The proof is of potential-theoretic type.
The derived equation is related to the {\it reduced derived equation}
\[ \frac{\partial \alpha}{\partial x} =
{\left({ \frac{\partial{(x \circ \varphi_{0})}}{\partial{x}} }\right)}^{-1}
\frac{\Delta}{\hat{u}_{0}^{2}} . \]
associated to ${\rm exp}(f \partial/\partial{x})$ and the family $\varphi_{\lambda,\Delta}$
where $\varphi_{0}= {\rm exp}(\hat{u}_{0} f \partial / \partial x)$.
The equation is linear in $\Delta$ and it is in general divergent.
\begin{pro}
\label{pro:exlodi}
Let $f \in \convn$ such that $f$ and $\partial{f}/\partial{x}$ vanish at the origin.
There exists $\varphi_{0} \in {\mathcal D}_{f}'$ such that
the infinitesimal generator $\log \varphi_{0}$ of $ \varphi_{0}$ is divergent.
\end{pro}
The previous result holds true even if $f=0$ contains fibered irreducible components.
The connection between the derived equation and the reduced derived equation
is provided by the following proposition.
\begin{pro}
\label{pro:linkder}
Fix $\varphi_{0} \in {\mathcal D}_{f}'$ and $\Delta \in \convn$.
Then the derived equation associated to ${\rm exp}(f \partial/\partial{x})$ and
$\varphi_{\lambda, \Delta}$ has a formal solution converging on $f=0$ if and only if
the reduced derived equation has a formal solution converging on $f=0$.
\end{pro}
The advantage of the reduced derived equation is that it
separates the roles of $\varphi_{0}$ and $\Delta$.

Let us focus on th. \ref{teo:intr2}. The proof of th. \ref{teo:intr3}
is analogous.  Let $f=x^{a}{(x-x_{1})}^{b}$.
We consider $\varphi_{0} \in {\mathcal D}_{f}' \subset \diff{up}{2}$
such that $\log \varphi_{0}$ is divergent. We define the operator
$S_{a,b,\varphi_{0}}:{\mathbb C}\{x,x_{1}\} \to {\mathbb C}[[x_{1}]]$ given by
\[ S_{a,b,\varphi_{0}}(\Delta) = \hat{\alpha}_{\Delta}(x_{1},x_{1}) - \hat{\alpha}_{\Delta}(0,x_{1}) \]
where $\hat{\alpha}_{\Delta}$ is a formal solution of the reduced derived equation associated to
${\rm exp}(x^{a}{(x-x_{1})}^{b}\partial/\partial{x})$ and $\varphi_{\lambda,\Delta}$.
If theorem \ref{teo:intr2} is false then
$S_{a,b,\varphi_{0}}({\mathbb C}\{x,x_{1}\}) \subset {\mathbb C}\{x_{1}\}$.
The following proposition leads us to a contradiction.
\begin{pro}
\label{pro:derlog}
Let $f=x^{a}{(x-x_{1})}^{b}$. Fix $\varphi_{0} \in {\mathcal D}_{f}'$.
Suppose that $S_{a,b,\varphi_{0}}({\mathbb C}\{x,x_{1}\}) \subset {\mathbb C}\{x_{1}\}$.
Then $\log \varphi_{0}$ converges.
\end{pro}
The proof of the last proposition is based on the uniform boundedness principle
and the properties of the Hilbert matrices.
\section{Formal properties of up-diffeomorphisms}
\label{sec:formal}
Let ${\rm Diff} ({\mathbb C}^{n+1},q)$ be the group of germs of complex analytic diffeomorphisms defined in
a neighborhood of $q \in {\mathbb C}^{n+1}$.
Consider coordinates $(x,x_{1},\hdots,x_{n}) \in {\mathbb C}^{n+1}$.
We define
\[ \diff{p}{n+1} = \{ \varphi \in \diff{}{n+1} : x_{j} \circ \varphi = x_{j} \
{\rm for \ any} \ 1 \leq j \leq n \}  . \]
We denote by $\diff{u}{n+1}$ the subgroup of unipotent elements of $\diff{}{n+1}$, more precisely
$\varphi \in \diff{u}{n+1}$ if $j^{1} \varphi$ is a unipotent linear isomorphism
(i.e. $j^{1} \varphi - Id$ is nilpotent).
By definition a unipotent parameterized diffeomorphism (up-diffeomorphism for shortness) is an
element of
\[ \diff{up}{n+1} \stackrel{def}{=} \diff{u}{n+1} \cap \diff{p}{n+1}. \]
Indeed up-diffeomorphisms are exactly the $n$-parameter unfoldings of tangent to the identity
diffeomorphisms. We denote
$\diffh{}{n+1}$, $\diffh{p}{n+1}$, $\diffh{u}{n+1}$ and $\diffh{up}{n+1}$
the formal completions of $\diff{}{n+1}$, $\diff{p}{n+1}$, $\diff{u}{n+1}$ and $\diff{up}{n+1}$
respectively.

  The unipotent germs of diffeomorphisms are related with nilpotent vector fields. We denote by
${\mathcal X} \cn{n+1}$ the set of germs of complex analytic
vector fields which are singular at $0$. We denote
by ${\mathcal X}_{N} \cn{n+1}$ the subset of ${\mathcal X} \cn{n+1}$ of nilpotent vector fields, i.e.
vector fields whose first jet has the unique eigenvalue $0$.
The formal completions of these spaces are denoted by
$\hat{\mathcal X} \cn{n+1}$ and $\hat{\mathcal X}_{N} \cn{n+1}$ respectively.

The expression
\begin{equation}
\label{equ:taylor2}
 {\rm exp} (t \hat{X}) = \left({
\sum_{j=0}^{\infty} \frac{t^{j}}{j!} \hat{X}^{j}(x),
\sum_{j=0}^{\infty} \frac{t^{j}}{j!} \hat{X}^{j}(x_{1}), \hdots, \sum_{j=0}^{\infty}
\frac{t^{j}}{j!} \hat{X}^{j}(x_{n}) }\right)
\end{equation}
defines the exponential of $t \hat{X}$ for $t \in {\mathbb C}$.
Let us remark that $\hat{X}^{j}(g)$ is the result of applying $j$ times the derivation
$\hat{X}$ to the power series $g$. The definition coincides
with the classical one if $\hat{X}$ is a germ of convergent vector field. For
$\hat{X}$ in $\hat{\mathcal X}_{N} \cn{n+1}$ the sums defining the components of ${\rm exp}(t \hat{X})$
converge in the Krull topology of ${\mathbb C}[[x,x_{1},\hdots,x_{n}]]$, i.e. the multiplicity at the origin of
$\hat{X}^{j}(g)$ tends to $\infty$ when $j \to \infty$ for any $g \in {\mathbb C}[[x,x_{1},\hdots,x_{n}]]$.
The next proposition is classical.
\begin{pro}
The exponential mapping ${\rm exp}$ induces a bijection from
$\hat{\mathcal X}_{N} \cn{n+1}$ onto $\diffh{u}{n+1}$.
\end{pro}
As we noted in the introduction the infinitesimal generator of a germ
of diffeomorphism is in general a divergent vector field.
\begin{defi}
Let $\varphi \in \diff{u}{n+1}$. We denote by $\log \varphi$ the unique element of  $\hat{\mathcal X}_{N} \cn{n+1}$
such that $\varphi = {\rm exp}(\log \varphi)$. We say that
$\log \varphi$ is the infinitesimal generator of $\varphi$.
\end{defi}
The previous proposition allows to transport the formal classification problem in $\diff{u}{n+1}$ to
$\hat{\mathcal X}_{N} \cn{n+1}$.
This is a simplification since $\hat{\mathcal X}_{N} \cn{n+1}$ can be interpreted as the
Lie algebra of $\diffh{u}{n+1}$.

Next we describe the nature of the infinitesimal generator of a up-diffeomorphism.
\begin{pro}
\label{pro:logfor} \cite{UPD}
Let $\varphi \in \diffh{up}{n+1}$. Then $\log \varphi$ is of the form
$\hat{u} (x \circ \varphi - x) \partial /\partial{x}$ where $\hat{u} \in {\mathbb C}[[x,x_{1},\hdots,x_{n}]]$
is a unit.
\end{pro}
This proposition is a consequence of the geometrical nature of the mappings ${\rm exp}$ and $\log$.
More precisely $\log \varphi$ is collinear to $\partial / \partial{x}$ since
$\varphi$ and then $\log \varphi$
preserve the leaves of the
foliation $dx_{1}= \hdots = dx_{n}=0$. Moreover the singular set of $\log \varphi$
coincides with $Fix (\varphi)$.

  The fixed points set $Fix (\varphi)$ of $\varphi \in \diff{up}{n+1}$ is a hypersurface.
Consider the determinant $|Jac \ \varphi|$ of the jacobian matrix of $\varphi$.
\begin{defi}
We say that
an irreducible component $\gamma$ of $Fix (\varphi)$ is unipotent if $|Jac \ \varphi|_{|\gamma} \equiv 1$.
\end{defi}
  An element $\varphi \in \diff{up}{n+1}$ is defined in some open neighborhood $U$ of the origin.
Thus $\varphi$ induces an element $\varphi_{q} \in {\rm Diff}_{p} ({\mathbb C}^{n+1},q)$
for any $q \in U \cap Fix (\varphi)$. Moreover $\varphi$ belongs to
${\rm Diff}_{up} ({\mathbb C}^{n+1},q)$ if $q$ belongs to a unipotent irreducible component
of $Fix (\varphi)$.
Since the mapping $\log$ is of
geometrical type it is natural to expect an extension of $\log \varphi$ to $Fix (\varphi)$.
Now, we introduce some definitions providing the context to describe this phenomenon.

  The formal completion of a complex space $(U,{\mathcal O}(U))$ ($U$ is a topological space
and ${\mathcal O}(U)$ is its sheaf of analytic functions) along a sub-variety $V$ given
by a sheaf of ideals $I$ is the space $(U,\hat{\mathcal O}_{I}(U))$ where
\[ \hat{\mathcal O}_{I}(U) = \lim_{\leftarrow} \frac{{\mathcal O}(U)}{{I}^{j}} . \]
Consider a series $\hat{g} \in {\mathbb C}[[x,x_{1},\hdots,x_{n}]]$ and
a germ of analytic variety $V \subset {\mathbb C}^{n+1}$ at $0$ given by an ideal
$I$. Then $\hat{g}$ is
\begin{itemize}
\item {\it transversally formal} along $V$ if
$\hat{g} \in \lim_{\leftarrow} \convn /I^{j}$.  \\
\item {\it uniformly transversally formal} along $V$ if $\hat{g} \in \hat{\mathcal O}_{I}(U)$
for some neighborhood of the origin $U$. \\
\end{itemize}
For shortness we say that $\hat{g}$ is t.f. along $V$ in the former case whereas
$\hat{g}$ is u.t.f. along $V$ in the latter case.
\begin{defi}
We say that $\hat{X} \in \hat{\mathcal X} \cn{n+1}$
is t.f. along $V$ if $\hat{X}(x)$, $\hat{X}(x_{1})$, $\hdots$, $\hat{X}(x_{n})$ are
t.f. along $V$. There is an analogous definition for $\hat{\sigma} \in \diffh{}{n+1}$
by considering $x \circ \hat{\sigma}$, $x_{1} \circ \hat{\sigma}$, $\hdots$,
$x_{n} \circ \hat{\sigma}$.
The definitions of uniform formal transversality for formal diffeomorphisms and vector fields
are analogous.
\end{defi}
\begin{defi}
\label{def:g}
Consider a set $W \subset {\mathbb C}^{n+1}$. We can define the ring $G_{W}$ of germs
of holomorphic functions defined in a neighborhood of $W$.
\end{defi}
Next lemma provides a handy characterization of u.t.f. functions.
\begin{lem}
\label{lem:handy}
\cite{UPD}
Fix $\hat{g} \in \formn$ and
a germ of analytic variety $V \subset {\mathbb C}^{n+1}$ at $0$ given by an ideal
$I$. Then $\hat{g}$ is u.t.f. along $V$ if and only if
$\hat{g} \in \lim_{\leftarrow} G_{V \cap W}/I^{j}$ for some neighborhood $W$ of the origin.
\end{lem}
We describe the t.f. behavior of
$\log \varphi$ for $\varphi \in \diff{up}{n+1}$.
\begin{pro}
\label{pro:logtra}
\cite{UPD}
Let $\varphi \in \diff{up}{n+1}$.
Then $\log \varphi$ is t.f. along $Fix (\varphi)$. Moreover $\log \varphi$ is u.t.f.
along $\gamma$ for every
unipotent irreducible component $\gamma$ of $Fix (\varphi)$.
\end{pro}
Let us focus next in the formal classification of up-diffeomorphisms.
Let $\varphi \in \diff{up}{n+1}$.
Denote $f=x \circ \varphi - x$ and $\hat{u} = (\log \varphi)(x)/f$.
We consider the dual form $\hat{\Omega}_{\varphi} = dx/(\hat{u} f)$.
By the uniqueness of the infinitesimal generator
we have that conjugating elements of $\diff{up}{n+1}$ is equivalent
to conjugating their dual forms.
By the previous proposition we can choose $u_{\varphi}$ in ${\mathbb C}\{x,x_{1},\hdots,x_{n}\}$
such that $\hat{u} - u_{\varphi} \in (f)$. We denote $\Omega_{\varphi} = 1/ (u_{\varphi}f)$.
The dual form $\hat{\Omega}_{\varphi}$ is of the form
\[ \hat{\Omega}_{\varphi} = \frac{dx}{u_{\varphi} f} +
\frac{u_{\varphi} - \hat{u}}{f} \frac{1}{u_{\varphi} \hat{u}} dx . \]
Since $\hat{\Omega}_{\varphi} - {\Omega}_{\varphi}$ does not have poles then
\begin{lem}
\label{lem:fornorr}
\cite{UPD}
Let $\varphi \in \diff{up}{n+1}$. Then the diffeomorphisms
${\rm exp}(u_{\varphi}(x \circ \varphi - x)\partial/\partial{x})$ and $\varphi$ are formally conjugated
by some $\hat{\sigma} \in \diffh{p}{n+1}$ such that
$x \circ \hat{\sigma} - x \in (x \circ \varphi - x)$.
\end{lem}
\begin{defi}
A germ of analytic set $\gamma \subset ({\mathbb C}^{n+1},0)$ is {\it fibered} if it is
the union of orbits of $\partial/\partial{x}$.
We say that $\hat{\sigma} \in \diffh{p}{n+1}$ is {\it normalized} with respect
to analytic set $\beta \subset ({\mathbb C}^{n+1},0)$ if
 $x \circ \hat{\sigma} - x \in I(\gamma)$ for any non-fibered irreducible component
 $\gamma$ of $\beta$.
\end{defi}
\begin{pro}
\label{pro:action}
\cite{UPD}
Let $\varphi_{1},\varphi_{2} \in \diff{up}{n+1}$ be formally conjugated. Then there exist
$\hat{\sigma} \in \diffh{p}{n+1}$  
and $\sigma \in \diff{}{n+1}$ such that
$(\hat{\sigma} \circ \sigma) \circ \varphi_{1} = \varphi_{2} \circ (\hat{\sigma} \circ \sigma)$
and $\hat{\sigma}$ is normalized with respect to $Fix (\varphi_{2})$.
\end{pro}
The proposition answers question (\ref{que:first}).
It implies that up to an analytic change of coordinates
the formal conjugations are normalized with respect to
$Fix (\varphi_{1})=Fix (\varphi_{2})$. Then we can
suppose that $\varphi_{1}$ and $\varphi_{2}$ belong to
\[ {\mathcal D}_{f} =\{ \varphi \in \diff{up}{n+1} : (x \circ \varphi -x)/f \ {\rm is \ a \ unit} \} \]
for some $f \in {\mathbb C}\{x,x_{1},\hdots,x_{n}\}$, for instance we can choose
$f=x \circ \varphi_{1}-x$.
The elements of ${\mathcal D}_{f}$ are the unfoldings whose fixed points set is $f=0$.
When conjugating elements of ${\mathcal D}_{f}$ normalized stands for normalized
with respect to $f=0$. We focus on formal normalized conjugations from now on.

We linearize the (normalized) formal conjugacy problem by
expressing the formal properties in terms of the infinitesimal generator.
Below we explain that the existence of a normalized formal conjugation is equivalent to
the existence of a meromorphic solution of an ordinary differential equation
with prescribed poles.

  Let $f \in {\mathbb C}\{x,x_{1},\hdots,x_{n}\}$.
Let $\prod_{j=1}^{p} f_{j}^{l_{j}} \prod_{k=1}^{s} F_{k}^{m_{k}}$ be the decomposition
of $f$ in irreducible factors; we suppose that $f_{j}=0$ is non-fibered for $1 \leq j \leq p$ whereas
$F_{k}=0$ is fibered for $1 \leq k \leq s$. We denote $f_{N} = \prod_{j=1}^{p} f_{j}^{l_{j}}$ and
$f_{F} = \prod_{k=1}^{s} F_{k}^{m_{k}}$. We choose $f_{F} \in {\mathbb C}\{x_{1},\hdots,x_{n}\}$.
The functions $f_{N}$ and $f_{F}$ are well defined up to
multiplicative units.

  Consider the equivalence relation ${\sim}_{*}$ in $ {\mathcal D}_{f}$ given by
$\varphi_{1} {\sim}_{*} \varphi_{2}$ if $\varphi_{1}$ and $\varphi_{2}$ are conjugated by
a normalized element $\hat{\sigma}$ of $\diffh{p}{n+1}$.
If there is a choice of $\hat{\sigma}$ such that $\hat{\sigma}$ is t.f. along $f=0$
we denote $\varphi_{1} \stackrel{t}{\sim}_{*} \varphi_{2}$.
If we can choose $\hat{\sigma}$ to be also u.t.f. along the unipotent components of $f=0$ we denote
$\varphi_{1} \stackrel{ut}{\sim}_{*} \varphi_{2}$. The classes of
the relation ${\sim}_{*}$ are connected in the compact-open topology. As a consequence
we can use the method of the path to obtain the invariants for
the equivalence relation ${\sim}_{*}$. Given
$\varphi_{1}, \varphi_{2} \in {\mathcal D}_{f}$ we define the {\it homological equation}
\[ \frac{\partial \alpha}{\partial{x}} = \hat{\Omega}_{\varphi_{1}} - \hat{\Omega}_{\varphi_{2}} =
\left({ \frac{1}{\hat{u}_{1}} - \frac{1}{\hat{u}_{2}} }\right) \frac{1}{f} \]
where $\log \varphi_{j} = \hat{u}_{j} f \partial / \partial x$
for $j \in \{1,2\}$. Consider the decomposition
$\prod_{j=1}^{p} f_{j}^{l_{j}}$ of $f_{N}$ in irreducible factors.
\begin{defi}
\label{def:spec}
We say that the homological equation associated to
$\varphi_{1}, \varphi_{2}  \in {\mathcal D}_{f}$ is special with respect to $f$ if there exists a
solution of the form $\hat{\beta} / (f_{F} \prod_{j=1}^{p} f_{j}^{l_{j}-1})$
where $\hat{\beta} \in {\mathbb C}[[x,x_{1},\hdots,x_{n}]]$.
This solution is called special with respect to $f$. A convergent special equation
$\partial{\alpha}/\partial{x}=K/f$ ($K \in \convn$) has a convergent special solution
\cite{UPD}.
\end{defi}
\begin{teo}
\cite{UPD}
\label{teo:UPD}
Let $\varphi_{1},\varphi_{2} \in {\mathcal D}_{f}$. Then $\varphi_{1} {\sim}_{*} \varphi_{2}$
if and only if the homological equation
associated to $\varphi_{1}$, $\varphi_{2}$ is special (with respect to $f$).
\end{teo}
Theorem \ref{teo:UPD} implies that
a complete system of invariants for the formal classification is obtained by
studying the obstruction for a homological equation to be special \cite{UPD}.

The techniques used in the proof of theorem \ref{teo:UPD} in \cite{UPD}
are a simplified version of
the ideas in the proof of theorem \ref{teo:cartra} later in the paper.

   The next proposition shows that the formal invariants of the germs induced
by a up-diffeomorphism $\varphi$ at its fixed points are basically also formal invariants
of $\varphi$. Philosophically the origin is not much different than any other fixed point.
Thus the property in question (\ref{que:second}) holds true for generic points.
\begin{pro}
\label{pro:prifp}
Let $\varphi_{1},\varphi_{2} \in {\mathcal D}_{f}$. Consider an irreducible component $\gamma$
of $f=0$. Then $\varphi_{1} {\sim}_{*} \varphi_{2}$ implies that $\varphi_{1}$ and $\varphi_{2}$
are conjugated by an element of $\widehat{\rm Diff}_{p}({\mathbb C}^{n+1},q)$ for a generic
$q \in \gamma$.
Moreover this property is satisfied for any
$q \in \gamma$ if $\gamma$ is a unipotent component of $Fix (\varphi_{1})$.
\end{pro}
For us the complementary of a proper real analytic set $S \subset \gamma$ is generic.
\begin{proof}
We have that $\varphi_{1} {\sim}_{*} \varphi_{2}$
implies that
$\partial{(x \circ \varphi_{1})}/\partial{x} \equiv \partial{(x \circ \varphi_{2})}/\partial{x}$
in $\gamma$. Suppose that $\gamma$ is non-unipotent with respect to $\varphi_{1}$;
since $\varphi_{1}$ is unipotent then
$\partial{(x \circ \varphi_{1})}/\partial{x}: \gamma \to {\mathbb C}$ is a non-constant function
whose value at the origin is $1$. By taking
$q \in \gamma \setminus \{ \partial{(x \circ \varphi_{1})}/\partial{x} \in {\mathbb S}^{1} \}$
we avoid the small divisors issues to obtain that $\varphi_{1}$ and $\varphi_{2}$ are conjugated
by an element of $\widehat{\rm Diff}_{p}({\mathbb C}^{n+1},q)$.

  Suppose that $\gamma$ is unipotent.
Let $\varphi_{j}={\rm exp}(\hat{u}_{j} f \partial / \partial x)$ for $j \in \{1,2\}$.
There exist units $u_{1}, u_{2} \in \convn$ such that
$\hat{u}_{j} - u_{j} \in (f)$ for $j \in \{1,2\}$ by prop. \ref{pro:logtra}.
The homological equation associated to $\varphi_{1}$ and $\varphi_{2}$ is of
the form
\[ \frac{\partial \alpha}{\partial x} =
\left( \frac{1}{{u}_{1} f} -  \frac{1}{{u}_{2} f} \right) + \hat{K} \]
where $\hat{K} \in \formn$ is u.t.f. along $\gamma$ by prop. \ref{pro:logtra}.
Since $\partial \alpha/\partial x = \hat{K}$ is obviously special then
$\partial \alpha/\partial x = 1/(u_{1}f) - 1/(u_{2}f)$ is special.
Moreover it has an analytic solution (lemma 5.5 in \cite{UPD}).
As a consequence
$\varphi_{1}$ and $\varphi_{2}$ are conjugated by an element of
$\widehat{\rm Diff}_{p}({\mathbb C}^{n+1},q)$
for any $q \in \gamma$ in a neighborhood of the origin by theorem \ref{teo:UPD}.
\end{proof}
\begin{rem}
\label{rem:nutf}
Let $\varphi=(x+y-x^{2},y) \in \diff{up}{2}$.
There exists $\tau = {\rm exp}(u_{\varphi}(y-x^{2}) \partial/\partial{x})$
such that $\varphi {\sim}_{*} \tau$ by lemma \ref{lem:fornorr}.
Now consider $x_{0}$ such that $1-2x_{0}$ is a root of the unit different than $1$.
The germs induced by $\tau_{|y=x_{0}^{2}}$ and $\varphi_{|y=x_{0}^{2}}$ at $x=x_{0}$
are not formally conjugated since the former one is periodic (it is embeddable)
and the latter one is not.
Hence $\tau$ and $\varphi$ are not conjugated by an element of
$\widehat{\rm Diff} ({\mathbb C}^{2},(x_{0},x_{0}^{2}))$.
Thus proposition \ref{pro:prifp} can not be improved to the whole $\gamma$
instead of a generic subset.
Clearly there is no u.t.f. (along $y=x^{2}$)
transformation conjugating $\varphi$ and $\tau$.
\end{rem}
\begin{rem}
Let $\varphi_{1},\varphi_{2} \in {\mathcal D}_{f}$ with $\varphi_{1} {\sim}_{*} \varphi_{2}$.
On the one hand remark \ref{rem:nutf}
implies that formal conjugations in general are not u.t.f. along non-unipotent irreducible components
of $f=0$. On the other hand the possibility remains open for unipotent components by
proposition \ref{pro:prifp}. This discrepancy justifies
why we do not require conjugating mappings to be
u.t.f. along non-unipotent components when defining the equivalence relation
$\stackrel{ut}{\sim}_{*}$.
\end{rem}
\section{Transversaly formal conjugations}
  In this section we provide a necessary and sufficient condition to assure
that two elements $\varphi_{1}, \varphi_{2}$ in ${\mathcal D}_{f}$
satisfy $\varphi_{1} \stackrel{t}{\sim}_{*} \varphi_{2}$.
The condition is stated in terms of the homological equation.
\begin{lem}
\label{lem:spedec}
Let $f = f_{F} \prod_{j=1}^{p} f_{j}^{l_{j}} \in {\mathbb C}\{x,x_{1},\hdots,x_{n}\}$.
Consider a special solution $\hat{\alpha}$ of a homological equation
$\partial{\alpha}/\partial{x} = \hat{K}/f$ where $\hat{K}$ in $\formn$ is t.f. along $f=0$.
Then $\hat{\alpha}$ is of the form
\[ \hat{\alpha} = \frac{\tau}{f_{F} \prod_{j=1}^{p} f_{j}^{l_{j}-1}} +
\frac{\hat{\beta}}{f_{F}} \]
where $\tau \in \convn$ and $\hat{\beta} \in \formn$.
\end{lem}
\begin{proof}
We have $\hat{K}/f  = K_{0}/f + \hat{K}_{1}$ for some
$K_{0} \in {\mathbb C}\{x,x_{1},\hdots,x_{n}\}$ and
$\hat{K}_{1} \in {\mathbb C}[[x,x_{1},\hdots,x_{n}]]$ by proposition \ref{pro:logtra}.
The special equation $\partial{\alpha}/\partial{x}=K_{0}/f$ has
a convergent solution $\alpha_{0} = \tau / (f_{F} \prod_{j=1}^{p} f_{j}^{l_{j}-1})$
(lemma 5.5 in \cite{UPD}).
There exists $\hat{\zeta} \in {\mathbb C}[[x,x_{1},\hdots,x_{n}]]$ such that
$\partial{\hat{\zeta}}/\partial{x} = \hat{K}_{1}$. We deduce that
$\partial({\hat{\alpha} - \alpha_{0} - \hat{\zeta})}/\partial{x}=0$. Since
\[ (\hat{\alpha} - \alpha_{0} - \hat{\zeta}) f_{F} \prod_{j=1}^{p} f_{j}^{l_{j}-1} \in \formn \]
then $\hat{\alpha} - \alpha_{0} - \hat{\zeta}$ is of the form
$\hat{\xi} /f_{F}$ where $\hat{\xi} \in {\mathbb C}[[x_{1},\hdots,x_{n}]]$.
Now we define $\hat{\beta} = f_{F} \hat{\zeta} + \hat{\xi}$.
\end{proof}
\begin{defi}
\label{def:conv}
Let $\varphi_{1},\varphi_{2} \in {\mathcal D}_{f}$ such that $\varphi_{1} {\sim}_{*} \varphi_{2}$.
Consider a special solution $\hat{\alpha}$ of the homological equation
associated to $\varphi_{1},\varphi_{2}$. By the previous lemma we have
$\hat{\alpha}= \alpha_{0} + \hat{\beta}/f_{F}$ where $\alpha_{0}$ is convergent and special
and $\hat{\beta}$ belongs to $\formn$. We say that $\hat{\alpha}$
{\it converges by restriction} to an irreducible analytic set
$\gamma \not \subset \{ f_{F}=0 \}$ if there exists $\kappa \in {\mathbb C}\{x,x_{1},\hdots,x_{n}\}$
such that $\hat{\beta} - \kappa \in I(\gamma)$.
We say that $\hat{\alpha}$ is t.f. (resp. u.t.f.) along an analytic set $\gamma$
if $\hat{\beta}$ is t.f. (resp. u.t.f.) along $\gamma$.
\end{defi}
Our goal in this section is proving:
\begin{teo}
\label{teo:cartra}
  Let $\varphi_{1}, \varphi_{2} \in {\mathcal D}_{f}$. Then
$\varphi_{1} \stackrel{t}{\sim}_{*} \varphi_{2}$ (resp. $\varphi_{1} \stackrel{ut}{\sim}_{*} \varphi_{2}$)
if and only if there exists a special solution of the homological equation
associated to $\varphi_{1},\varphi_{2}$ which is t.f. (resp. u.t.f.) along $f_{F}=0$
and converges by restriction to $f_{N}=0$.
\end{teo}
The next lemma is the first step in the proof.
It implies that convergence and formal transversality along
generically transverse irreducible components of the fixed points set are equivalent
properties for special solutions.
\begin{lem}
\label{lem:spsotra}
Let $f = f_{F} \prod_{j=1}^{p} f_{j}^{l_{j}} \in {\mathbb C}\{x,x_{1},\hdots,x_{n}\}$. Fix $1 \leq j \leq p$.
Consider a special solution $\hat{\alpha}$ of a
homological equation $\partial{\alpha}/\partial{x} = \hat{K}/f$
($\hat{K} \in {\mathbb C}[[x,x_{1},\hdots,x_{n}]]$)
such that $\hat{\alpha}$ converges by restriction to $f_{j}=0$.
Assume that $\hat{K}$ is t.f. (resp. u.t.f.)
along $f_{j}=0$. Then $\hat{\alpha}$ is t.f. (resp. u.t.f.) along $f_{j}=0$.
\end{lem}
\begin{proof}
Suppose $\hat{K}$ is u.t.f. along $f_{j}=0$, the t.f. case is simpler.
By lemma \ref{lem:spedec} the solution $\hat{\alpha}$ is of the form
$\tau/(f_{F} \prod_{j=1}^{p} f_{j}^{l_{j}-1}) + \hat{\beta}/f_{F}$. By the proof of
lemma \ref{lem:spedec} the series $\hat{\beta}$ satisfies
$\partial{(\hat{\beta}/f_{F})}/\partial{x} = \hat{K}_{1}$ for some $\hat{K}_{1}$ in $\formn$.
Moreover $\hat{K}_{1}$ is u.t.f. along $f_{j}=0$.

Let us explain the idea of the proof. By hypothesis there exists
$\beta_{1} \in {\mathbb C}\{x,x_{1},\hdots,x_{n}\}$ such that
$\hat{\beta} - \beta_{1}$ belongs to the ideal $(f_{j})$. It suffices to prove
the existence of a neighborhood $W$ of $0$ in ${\mathbb C}^{n+1}$
such that given $\beta_{k} \in {\mathcal O}(W)$ with
$\hat{\beta} - \beta_{k} \in (f_{j}^{k})$ we can find
$\beta_{k+1}  \in {\mathcal O}(W)$ satisfying $\hat{\beta} - \beta_{k+1} \in (f_{j}^{k+1})$.
We show that $(\hat{\beta} - \beta_{k}) / f_{j}^{k}$ converges by restriction to $f_{j}=0$.
We can obtain $\beta_{k+1}$ of the form $\beta_{k}+ f_{j} \xi$
where $\xi$ is the Weierstrass remainder of dividing
 $(\hat{\beta} - \beta_{k}) / f_{j}^{k}$ by $f_{j}$.

Consider coordinates $(y,y_{1},\hdots,y_{n})$ in ${\mathbb C}^{n+1}$ such that
$f_{j}=0$ does not contain $y_{1}=\hdots=y_{n}=0$.
We denote $\nu = \nu(f_{j}(y,0,\hdots,0))$ the multiplicity at $y=0$. Up to a multiplicative unit
$f_{j}$ can be expressed in the Weierstrass form
\[ y^{\nu} + a_{\nu-1}(y_{1},\hdots,y_{n})  y^{\nu-1}+ \hdots + a_{0}(y_{1},\hdots,y_{n}) \]
where $a_{j} \in {\mathbb C}\{y_{1},\hdots,y_{n}\}$ for any $0 \leq j < \nu$.
We denote by
\[ \eta_{1}(y_{1}^{0},\hdots,y_{n}^{0}) , \hdots, \eta_{\nu}(y_{1}^{0},\hdots,y_{n}^{0}) \]
the $\nu$ points (counted with multiplicity) in
$\{ f_{j}=0 \} \cap \cap_{k=1}^{\nu} \{ y_{k}=y_{k}^{0} \}$.
We define
\[ \Delta(y_{1},\hdots,y_{n}) = \left({
\prod_{k=1}^{\nu} \left({ \frac{\partial{f_{j}}}{\partial{x}} \circ \eta_{k} }\right)
\prod_{1 \leq k < l  \leq \nu} {(\eta_{k}-\eta_{l})}^{2}
}\right) (y_{1},\hdots,y_{n}) . \]
The function $\Delta$ is continuous and holomorphic outside of $\Delta=0$;
thus $\Delta$ belongs to ${\mathbb C}\{y_{1},\hdots,y_{n}\}$. We choose a
domain $W= U \times V \subset {\mathbb C} \times {\mathbb C}^{n}$ in
coordinates $(y,y_{1},\hdots,y_{n})$
such that $\Delta \in {\mathcal O}(V)$ and
\[ \left\{{
\begin{array}{l}
a = f_{j} b \ {\rm where} \ a \in {\mathcal O}(W) \ {\rm and} \
b \in {\mathbb C}\{y,y_{1},\hdots,y_{n}\} \implies b \in {\mathcal O}(W) \\
a = \Delta b \ {\rm where} \ a \in {\mathcal O}(V) \ {\rm and} \
b \in {\mathbb C}\{y_{1},\hdots,y_{n}\} \implies b \in {\mathcal O}(V) .
\end{array}
}\right. \]
For a more detailed discussion on the existence of $V$ and $W$ see
section 4.2 of \cite{UPD}.
By shrinking $U$ and $V$ if necessary we can suppose that there
exist $\beta_{1} \in {\mathcal O}(W)$ and a sequence $K_{1,k} \in {\mathcal O}(W)$
such that $\hat{\beta}-\beta_{1} \in (f_{j})$ and $\hat{K}_{1} - K_{1,k} \in (f_{j}^{k})$
for any $k \in {\mathbb N}$.

Now we prove that given
$\beta_{k} \in {\mathcal O}(W)$ ($k \geq 1$)
such that $\hat{\beta} - \beta_{k} \in (f_{j}^{k})$ there exists $\beta_{k+1} \in {\mathcal O}(W)$
such that $\hat{\beta} - \beta_{k+1} \in (f_{j}^{k+1})$.
  We have $\hat{\beta} = \beta_{k} + f_{j}^{k} \hat{\xi}$ where $\hat{\xi} \in \formn$. We obtain
\[ \frac{\partial{\beta_{k}}}{\partial{x}} +
k f_{j}^{k-1} \frac{\partial{f_{j}}}{\partial{x}} \hat{\xi} - f_{F} K_{1,k} \in (f_{j}^{k}) . \]
Therefore $f_{F} K_{1,k} - \partial{\beta_{k}}/\partial{x} \in {\mathcal O}(W)$
belongs to the ideal $(f_{j}^{k-1})$. By the choice of $W$ we obtain that
$L \stackrel{def}{=} (f_{F} K_{1,k} - \partial{\beta_{k}}/\partial{x})/f_{j}^{k-1}$ belongs
to ${\mathcal O}(W)$. Moreover $k (\partial{f_{j}}/\partial{x}) \hat{\xi} - L$ belongs to $(f_{j})$.
We define
\[ \xi =
\frac{1}{k} \sum_{l=1}^{\nu} \frac{L}{\partial{f_{j}}/\partial{x}} \circ \eta_{l}
\prod_{m \in \{1,\hdots,\nu\} \setminus \{l\}} \frac{y-\eta_{m}}{\eta_{l} - \eta_{m}}  \]
for $(y,y_{1},\hdots,y_{n}) \in {\mathbb C} \times V$.
The meromorphic function $\xi$ is a polynomial of degree at most $\nu-1$ in the variable $y$.
Moreover $\xi \Delta$ is holomorphic in a neighborhood of the origin. Since
we have $\xi k \Delta (\partial{f_{j}}/\partial{x}) - L \Delta \in (f_{j})$ by definition of $\xi$ then
$\hat{\xi} \Delta - \xi \Delta$ belongs to $(f_{j})$. Consider the unique element
$\xi' \in {\mathbb C}[y][[y_{1},\hdots,y_{n}]]$ such that $\deg_{y} \xi' \leq \nu-1$ and
$\hat{\xi} - \xi' \in (f_{j})$. Then we have $\xi' \Delta \equiv \xi \Delta$ in
${\mathbb C}[[y,y_{1},\hdots,y_{n}]]$.
Thus the coefficients $\xi_{l} \in {\mathcal O}(V)$ of
$\xi \Delta = \sum_{l=0}^{\nu-1} \xi_{l} {y}^{l}$ belong to $(\Delta)$ for any $0 \leq l < \nu$.
By the choice of $V$ the function $\xi$ belongs to ${\mathcal O}(W)$.
Moreover $\xi$ clearly satisfies $\hat{\xi} - \xi \in (f_{j})$, thus we can define
$\beta_{k+1}=\beta_{k} + f_{j} \xi$.
\end{proof}
Now we can adapt the proof of theorem \ref{teo:UPD} that can be found in \cite{UPD}
to prove the sufficient condition of the theorem \ref{teo:cartra}.
\begin{proof}[proof of the sufficient condition in theorem \ref{teo:cartra}]
Let $\hat{\alpha}$ be a special solution of the homological equation associated to
$\varphi_{1}$ and $\varphi_{2}$. Suppose that $\hat{\alpha}$ converges by restriction to
$f_{N}=0$ and that $\hat{\alpha}$ is u.t.f. along $f_{F}=0$ (the t.f. case is analogous).
By lemma \ref{lem:spsotra} and proposition \ref{pro:logtra} the solution $\hat{\alpha}$
is t.f. along $f_{N}=0$ and u.t.f. along the unipotent irreducible components of $f_{N}=0$.

Let us use the path method  (see \cite{Rou:ast}, \cite{Mar:ast}).
The infinitesimal generator $\log \varphi_{j}$ of $\varphi_{j}$ is of the form 
$\hat{u}_{j} f \partial /\partial x$. We define
\[ X_{1+z} =  \frac{\hat{u}_{1} \hat{u}_{2}}{z \hat{u}_{1} + (1-z) \hat{u}_{2}}
f \frac{\partial}{\partial{x}}  \]
for $z \in {\mathbb C}$. The power series
$\hat{u}_{1} \hat{u}_{2}/(z \hat{u}_{1} + (1-z) \hat{u}_{2})$ is a unit for
$z \neq \hat{u}_{2}(0)/(\hat{u}_{2}(0)-\hat{u}_{1}(0))$.
The homological equation associated to ${\rm exp}(X_{1})$ and
${\rm exp}(X_{1+z})$ is of the form
$\partial \alpha / \partial x = z \hat{K} / f$ where
$\hat{K}$ belongs to ${\mathbb C}[[x,x_{1},\hdots,x_{n}]]$. They are all special since
the one corresponding to $z=1$ is. We have that
\[ \left[
\hat{\alpha} \frac{\hat{u}_{1} \hat{u}_{2}}{z \hat{u}_{1} + (1-z) \hat{u}_{2}}
f \frac{\partial}{\partial{x}} + \frac{\partial}{\partial{z}} ,
 \frac{\hat{u}_{1} \hat{u}_{2}}{z \hat{u}_{1} + (1-z) \hat{u}_{2}}
f \frac{\partial}{\partial{x}}
\right] \equiv 0. \]
Suppose $\hat{u}_{2}(0)/(\hat{u}_{2}(0)-\hat{u}_{1}(0)) \not \in [0,1]$. Then
\[ \hat{\sigma} =
{\rm exp} \left({ \hat{\alpha} \frac{\hat{u}_{1} \hat{u}_{2}}{z \hat{u}_{1} + (1-z) \hat{u}_{2}}
f \frac{\partial}{\partial{x}} + \frac{\partial}{\partial{z}}
}\right)_{|z=0} \]
is a normalized element of $\diffh{p}{n+1}$ conjugating $\varphi_{1}$ and $\varphi_{2}$
(see prop. 5.10 in \cite{UPD} for more details).
Since $\hat{\alpha}$, $\hat{u}_{1}$ and $\hat{u}_{2}$ are t.f. along $f=0$ then
$\hat{\sigma}$ is t.f. along $f=0$. Denote by $V$ the union of the unipotent components of $f=0$.
The expression for $\hat{\sigma}$ implies that
$x \circ \hat{\sigma} \in \lim_{\leftarrow} G_{V \cap W}/I(V)^{j}$ (see def. \ref{def:g})
for some open neighborhood $W$ of the origin. By lemma \ref{lem:handy}
we obtain that $\hat{\sigma}$ is u.t.f. along the unipotent components of $f=0$.

Suppose $\hat{u}_{2}(0)/(\hat{u}_{2}(0)-\hat{u}_{1}(0)) \in [0,1]$. Proceeding as
previously we obtain
${\rm exp}(X_{1}) \stackrel{ut}{\sim}_{*} {\rm exp}(X_{1+i})$ and
${\rm exp}(X_{1+i}) \stackrel{ut}{\sim}_{*} {\rm exp}(X_{2})$.
This implies $\varphi_{1} \stackrel{ut}{\sim}_{*} \varphi_{2}$.
\end{proof}
Let us introduce the tools to prove the necessary condition in theorem \ref{teo:cartra}.
Take a holomorphic function $\psi$ defined in some simply connected open subset $W$ of
$U \setminus \{ f=0\}$ for some open neighborhood $U$ of the origin. We say that
$\psi$ is an integral of the time form of $f \partial/\partial{x}$ if
$\psi \circ {\rm exp}(t f \partial/\partial{x}) = \psi +t$ for $t \in {\mathbb C}$.
This condition is equivalent to
\[ f \frac{\partial{\psi}}{\partial{x}}=1 \Leftrightarrow \frac{\partial{\psi}}{\partial{x}} = \frac{1}{f} .\]
The last condition can be used to extend $\psi$ along any continuous path
$\gamma:[0,1] \to [\cap_{j=1}^{n} \{ x_{j}=x_{j}^{0} \}] \cap [U \setminus \{ f=0 \}]$ such that $\gamma(0) \in W$.
In fact the value $\psi(\gamma(1))-\psi(\gamma(0))$ does not depend on the choice of $\psi$
but only on $f \partial / \partial{x}$ and $\gamma$.

Let $\prod_{j=1}^{p} f_{j}^{l_{j}}$ be the irreducible decomposition of $f_{N}$.
Next we explain how we can define $\psi \circ \hat{\sigma} - \psi$ for any $\hat{\sigma} \in \diffh{p}{n+1}$
such that $x \circ \hat{\sigma} - x \in (\prod_{j=1}^{p} f_{j})$.
Fix $k \in {\mathbb N}$. We consider
\[ \sigma_{k}=(x+ A_{k} \prod_{j=1}^{p} f_{j}, x_{1}, \hdots , x_{n}) \]
where $A_{k} \in \convn$ and
$x \circ \hat{\sigma} - x \circ \sigma_{k} \in \tilde{\mathfrak m}^{k}$
($\tilde{\mathfrak m}$ is the maximal ideal of $\formn$).

Let $\eta : [0,1] \to {\mathbb C}$ be a path admitting a holomorphic extension
$\eta: U \to {\mathbb C}$ to a neighborhood of $[0,1]$ and such that
$\eta(0)=0$, $\eta(1)=1$ and
$1+ \eta(\lambda)((\partial (x \circ \hat{\sigma})/\partial x)(0)-1) \neq 0$
for any $\lambda \in [0,1]$. We define
\[ \beta(x,x_{1},\hdots,x_{n}, \lambda) = (x +
\eta(\lambda)(x \circ \sigma_{k}(x,x_{1},\hdots,x_{n})-x),x_{1},\hdots,x_{n},\lambda) . \]
We have $\beta \in {\rm Diff} ({\mathbb C}^{n+2}, (0,\hdots,0,\lambda))$
for any $\lambda \in [0,1]$ by the choice of $\eta$ and the inverse function theorem.
By construction $f=0$ is invariant by $\beta$, therefore the path
$\beta (q,\lambda) \not \in \{ f=0 \}$ for all
$q \in U \setminus \{f=0\}$ and $\lambda \in [0,1]$.

Let $q \not \in \{ f=0 \}$, we define
\[ (\psi \circ \sigma_{k} - \psi)(q) = \psi_{1}(\sigma_{k}(q)) - \psi_{0}(q) \]
where $\psi_{0}$ is a holomorphic integral of the time form of $f \partial/\partial{x}$
in the neighborhood of $q$ and $\psi_{1}$ is the analytic continuation of $\psi_{0}$
along the path $\beta_{q}:[0,1] \to {\mathbb C}^{n+1} \setminus \{ f=0 \}$ given by
$\beta_{q}(\lambda)=\beta(q,\lambda)$.

We have
\[ \psi \circ \sigma_{k} - \psi = \int_{\beta} \frac{\partial \psi}{\partial x} dx =
O \left( \frac{1}{f}  (x \circ \sigma_{k} -x) \right) =
O \left( \frac{A_{k}}{ f_{F} \prod_{j=1}^{p} f_{j}^{l_{j}-1} } \right)    \]
and
\begin{equation}
\label{equ:tf}
 (\psi \circ \sigma_{k} - \psi \circ \sigma_{l})(q)=
\int_{(x \circ \sigma_{l})(q)}^{(x \circ \sigma_{l})(q)} \frac{\partial \psi}{\partial x} dx =
O \left( \frac{x \circ \sigma_{k} -x \circ \sigma_{l}}{f}  (q) \right)      .
\end{equation}

Therefore
$(\psi \circ \sigma_{k} - \psi) f_{F} \prod_{j=1}^{p} f_{j}^{l_{j}-1} \in \convn$
converges in the Krull topology to a series we denote by
$(\psi \circ \hat{\sigma} - \psi) f_{F} \prod_{j=1}^{p} f_{j}^{l_{j}-1}$.
\begin{proof}[proof of the necessary condition in theorem \ref{teo:cartra}]
Suppose that $\varphi_{1} \stackrel{ut}{\sim}_{*} \varphi_{2}$
(the case $\varphi_{1} \stackrel{t}{\sim}_{*} \varphi_{2}$ is simpler).
Consider a normalized
$\hat{\sigma} \in \diffh{p}{n+1}$ conjugating $\varphi_{1}$ and $\varphi_{2}$
and such that $\hat{\sigma}$ is t.f. along $f=0$
and u.t.f. along the unipotent components of $f=0$.

  Let $\hat{u}_{j} = (\log \varphi_{j})(x)/f$ for $j \in \{1,2\}$.
We choose $u \in \convn$ such that $\hat{u}_{2} - u \in (f)$.
Denote $\hat{K}=1/(\hat{u}_{2}f) - 1/(uf)$, we have that $\hat{K} \in \formn$ is u.t.f. along $f_{F}=0$
by proposition \ref{pro:logtra}. Hence we have $\hat{K} = \sum_{j=0}^{\infty} K_{j} f_{F}^{j}$
where $K_{j} \in {\mathcal O}(U)$ for any $j \geq 0$ and some open neighborhood $U$ of $0$.
For a good choice of $U$ there exists $\gamma_{j} \in {\mathcal O}(U)$ such that
$\partial{\gamma_{j}}/\partial{x}=K_{j}$ for any $j \in {\mathbb N} \cup \{0\}$.
Now $\hat{\gamma}=\sum_{j=0}^{\infty} \gamma_{j} f_{F}^{j}$ is u.t.f. along $f_{F}=0$ and
it satisfies $\partial{\hat{\gamma}}/\partial{x}=\hat{K}$.

Consider an integral $\psi$ of the time form of $u f \partial/\partial{x}$. We have that
$\partial{(\psi + \hat{\gamma})}/\partial{x}=1/(\hat{u}_{2}f)$. Since
$\hat{\sigma}$ conjugates $\log \varphi_{1}$ and $\log \varphi_{2}$ we obtain
$\partial{((\psi + \hat{\gamma}) \circ \hat{\sigma})}/\partial{x}=1/(\hat{u}_{1}f)$
where $\psi \circ \hat{\sigma} = \psi + (\psi \circ \hat{\sigma} - \psi)$. Therefore
\[ \hat{\alpha} = (\psi + \hat{\gamma}) \circ \hat{\sigma} - (\psi + \hat{\gamma}) \]
is a solution of the homological equation associated to $\varphi_{1}$ and $\varphi_{2}$.
Moreover $\hat{\alpha}$ is special since $\psi \circ \hat{\sigma} - \psi$ is special.
By the choice of $\hat{\gamma}$ we obtain that $\hat{\gamma} \circ \hat{\sigma} - \hat{\gamma}$
is u.t.f. along $f_{F}=0$. It also converges by restriction to $f_{N}=0$ since
$\hat{\gamma} \circ \hat{\sigma} - \hat{\gamma} \in \sqrt{(f_{N})}$.
Let $V$ be the union of the unipotent components of $f=0$.
Consider $\sigma_{k} \in \diff{p}{n+1}$ such that
$x \circ \hat{\sigma} - x \circ \sigma_{k} \in (f^{k})$ for any $k \geq 2$.
By hypothesis we can suppose that for any $k \geq 2$ the diffeomorphism
$\sigma_{k}$ is defined in the neighborhood of $V \cap W$ for some open neighborhood $W$
of $0$ independent of $k$. The function
$(\psi \circ \sigma_{k} - \psi) f_{F} \prod_{j=1}^{p} f_{j}^{l_{j}-1}$
belongs to $G_{V \cap W}$ for any $k \geq 2$.
The equation (\ref{equ:tf}) implies that
$(\psi \circ \hat{\sigma} - \psi) f_{F} \prod_{j=1}^{p} f_{j}^{l_{j}-1} $
is t.f. and also u.t.f. along the unipotent components of $f=0$.
Now we obtain that
\[ \hat{\alpha} = (\psi  \circ \hat{\sigma} - \psi) +
(\hat{\gamma} \circ \hat{\sigma} - \hat{\gamma}) \]
is u.t.f. along $f_{F}=0$ and converges by restriction to $f_{N}=0$.
\end{proof}
We claim that in most situations theorem  \ref{teo:cartra}
can be improved and
we have  $\varphi_{1} \stackrel{ut}{\sim}_{*} \varphi_{2}$ if and only if there exists
a special solution of the homological equation converging by restriction to
$f_{N}=0$. Up to substracting an special analytic equation we can suppose that
the homological equation is of the form
\[ \frac{\partial \alpha}{\partial x} = \sum_{j=0}^{\infty} K_{j} f_{F}^{j} \]
where $K_{j} \in {\mathcal O}(U)$ for any $j \geq 0$ and some open neighborhood $U$ of $0$.
Suppose that $f_{N}=0$ is empty.
We can consider an analytic solution $\alpha_{j} \in  {\mathcal O}(U)$ of
$\partial \alpha/\partial x = K_{j}$ for any $j \geq 0$.
Then $\tilde{\alpha}=\sum_{j=0}^{\infty} \alpha_{j} f_{F}^{j}$ is a u.t.f. special solution of the
homological equation.

Suppose that $f_{N}=0$ contains a smooth hypersurface $\gamma$
given  by an equation $x = h(x_{1},\hdots,x_{n})$.  We require
$\alpha_{j}$ to satisfy $(\alpha_{j})_{|\gamma} \equiv 0$  for any $j \geq 0$
to obtain a solution $\tilde{\alpha}$ that is u.t.f. along $f_{F}=0$ and vanishing
by restriction to $\gamma$. Any other special solution of the homological
equation converging by restriction to $\gamma$ is of the form
$\tilde{\alpha}+ \xi(x_{1},\hdots.x_{n})/f_{F}$
where $\xi \in {\mathbb C}\{x_{1},\hdots,x_{n}\}$. Hence the existence of a special
solution converging by restriction to $f_{N}=0$ implies that $\tilde{\alpha}$
converges by restriction to $f_{N}=0$ and is u.t.f. along $f_{F}=0$.
We can apply theorem \ref{teo:cartra}.

The examples above do not describe all the situations where convergence of
a special solution by restriction to $f_{N}=0$ implies u.t.f. behavior along $f_{F}=0$.
Anyway by the previous discussion the counterexamples provided by the
Main Theorem and theorems  \ref{teo:intr2} and  \ref{teo:intr3} can only obtained if
$f_{N}=0$ contains the line $l \equiv \{x_{1}=\hdots=x_{n}=0\}$
or otherwise if $f_{N}=0$ neither contains $l$ nor it is a smooth hypesurface
transversal to $\partial / \partial x$. Both these settings represent obstructions
to the existence of t.f. conjugations (see sections
\ref{sec:bad} and \ref{sec:tra} respectively).
\section{Polynomial families}
Fix $f \in \convn$ such that $f(0)=\partial{f}/\partial{x}(0)=0$. We define
\[ {\mathcal D}_{f}' =\{ \varphi \in {\mathcal D}_{f} :
x \circ \varphi -  x \circ {\rm exp}(f \partial / \partial{x}) \in (f^{2}) \} . \]
The family ${\mathcal D}_{f}'$ is a subset of ${\mathcal D}_{f}$ whose elements
belong to the same class of formal conjugacy.
In fact given $\varphi_{1}, \varphi_{2} \in {\mathcal D}_{f}'$ the homological
equation associated to $\varphi_{1}$, $\varphi_{2}$ is special since it is of the form
$\partial \alpha/\partial x = \hat{K}$ for some
$\hat{K} \in {\mathbb C}[[x,x_{1},\hdots,x_{n}]]$.

We consider $\varphi_{0} \in {\mathcal D}_{f}'$; we define the polynomial family
\[ \varphi_{\lambda,\Delta} = (x \circ \varphi_{0} + \lambda {f}^{2} \Delta, x_{1}, \hdots,x_{n} ) \]
where $\Delta$ belongs to $\convn$.

The idea is using potential theory to obtain necessary conditions
in order to have
${\rm exp}(f \partial/\partial{x}) \stackrel{ut}{\sim}_{*} \varphi_{\lambda,\Delta}$
for any $\lambda \in {\mathbb C}$. These conditions are obtained by
derivating the homological equation with respect to $\lambda$.

Next we show that the homological equation associated to
${\rm exp}(f \partial/\partial{x})$ and $\varphi_{\lambda,\Delta}$ is polynomial in $\lambda$.
This is key to prove that either  we have
${\rm exp}(f \partial/\partial{x}) \stackrel{ut}{\sim}_{*} \varphi_{\lambda,\Delta} \ \forall
\lambda \in {\mathbb C}$ or
${\rm exp}(f \partial/\partial{x}) \not \stackrel{ut}{\sim}_{*} \varphi_{\lambda,\Delta}$
for generic $\lambda$.
\begin{lem}
Fix $\varphi_{0} \in {\mathcal D}_{f}'$ and $\Delta \in \convn$.
  The infinitesimal generator $\log \varphi_{\lambda, \Delta}$ is of the form
\begin{equation}
\label{equ:lempf1}
\log \varphi_{\lambda, \Delta} = f \left({ \sum_{0 \leq j_{0},\hdots,j_{n}}
a_{j_{0},\hdots,j_{n}}(\lambda) {x}^{j_{0}} \prod_{k=1}^{n} x_{k}^{j_{k}} }\right)
\frac{\partial}{\partial{x}}
\end{equation}
where $a_{j_{0},\hdots,j_{n}}$ is an entire function for any $0 \leq j_{0},\hdots,j_{n}$.
\end{lem}
The proof of the previous lemma is straightforward by using undetermined coefficients.
We can give a more precise description of the nature of the coefficients $a_{j_{0},\hdots,j_{n}}$.
\begin{lem}
\label{lem:bdddeg}
The coefficient $a_{j_{0},\hdots,j_{n}}$ is a polynomial whose degree is less or equal than
$j_{0} + \hdots + j_{n}$ for any $0 \leq j_{0},\hdots,j_{n}$.
\end{lem}
\begin{proof}[Sketch of proof]
We consider the diffeomorphism
\[ \tau_{\lambda} = (x/\lambda, x_{1}/\lambda, \hdots , x_{n}/\lambda) \circ
\varphi_{1/\lambda, \Delta} \circ (\lambda x, \lambda x_{1}, \hdots , \lambda x_{n}) . \]
Denote $\nu$ the multiplicity of $f$ at $0$.
Let $\tilde{f}=f(\lambda x, \lambda x_{1} , \hdots ,\lambda x_{n})/\lambda^{\nu}$. We have
\[ \left({ \frac{x \circ \tau_{\lambda} - x}{\lambda^{\nu-1} \tilde{f}} }\right) (0,\hdots,0)  =
\left({  \frac{x \circ \varphi_{0} - x}{f} }\right) (0,\hdots,0) \]
for any $\lambda \in {\mathbb C}$. Again by using undetermined coefficients we obtain
\[ \log \tau_{\lambda} = \lambda^{\nu-1} \tilde{f}
\left({
\sum_{0 \leq j_{0},\hdots,j_{n}}
b_{j_{0},\hdots,j_{n}}(\lambda) {x}^{j_{0}} \prod_{k=1}^{n} x_{k}^{j_{k}}
}\right) \]
where $b_{j_{0},\hdots,j_{n}}$ is an entire function for $0 \leq j_{0},\hdots,j_{n}$.
This implies that
\[ a_{j_{0},\hdots,j_{n}} ( 1 /\lambda) \lambda^{j_{0}+\hdots+j_{n}}
= b_{j_{0},\hdots,j_{n}}(\lambda) \]
for any $0 \leq j_{0},\hdots,j_{n}$. Since the functions involved are entire functions then
$a_{j_{0},\hdots,j_{n}}$ is a polynomial such that $\deg a_{j_{0},\hdots,j_{n}} \leq \sum_{k=0}^{n} j_{k}$
for any $0 \leq j_{0},\hdots,j_{n}$.
\end{proof}
\begin{pro}
\label{pro:polbnd}
Fix $\varphi_{0} \in {\mathcal D}_{f}'$ and $\Delta \in \convn$.
Then the homological associated to ${\rm exp}(f \partial/\partial{x})$
and $\varphi_{\lambda, \Delta}$ is of the form
\[\partial{\alpha}/\partial{x} = \sum_{0 \leq j_{0},\hdots,j_{n}} d_{j_{0},\hdots,j_{n}}
(\lambda){x}^{j_{0}} \prod_{k=1}^{n} x_{k}^{j_{k}} \]
where $d_{j_{0},\hdots,j_{n}} \in {\mathbb C}[\lambda]$ satisfies
$\deg d_{j_{0},\hdots,j_{n}} \leq j_{0} + \hdots + j_{n}+\nu(f)$ for any $0 \leq j_{0},\hdots,j_{n}$
($\nu (f)$ is the multiplicity of $f$ at $0$).
\end{pro}
\begin{proof}
As a consequence of lemma \ref{lem:bdddeg} the homological equation
\[ \frac{\partial{\alpha}}{\partial{x}} =
\left({ \frac{(\log \varphi_{\lambda, \Delta})(x)/f -1}{(\log \varphi_{\lambda, \Delta})(x)/f} }\right) \frac{1}{f} \]
associated to ${\rm exp}(f \partial/\partial{x})$ and $\varphi_{\lambda, \Delta}$ is of the form
\[ \frac{\partial{\alpha}}{\partial{x}}=
\frac{\sum  c_{j_{0},\hdots,j_{n}}(\lambda){x}^{j_{0}} \prod_{k=1}^{n} x_{k}^{j_{k}}  }{f}
= \sum_{0 \leq j_{0},\hdots,j_{n}} d_{j_{0},\hdots,j_{n}}(\lambda){x}^{j_{0}} \prod_{k=1}^{n} x_{k}^{j_{k}}\]
since $(\log \varphi_{\lambda,\Delta})(x) - f \in ({f}^{2})$; this is a consequence of
$\varphi_{\lambda,\Delta} \in {\mathcal D}_{f}'$ for $\lambda \in {\mathbb C}$.
By lemma \ref{lem:bdddeg} we deduce that   $d_{j_{0},\hdots,j_{n}}$ is a polynomial such that
$\deg d_{j_{0},\hdots,j_{n}} \leq \nu(f) + \sum_{k=0}^{n} j_{k}$.
\end{proof}
The structure of the homological equation makes useful the next theorem.
\begin{pro}[\cite{PM2}]
\label{pro:PM}
Let
\[
\hat{P} = \sum_{0 \leq j_{1},\hdots,j_{m}} P_{j_{1},\hdots,j_{m}}(\lambda) \prod_{k=1}^{m} y_{j}^{j_{k}}
\]
where $P_{j_{1},\hdots,j_{m}} \in {\mathbb C}[\lambda]$ and
$\deg P_{j_{1},\hdots,j_{m}} \leq A\sum_{k=1}^{m} j_{k} +B$ for some $A, B$ in ${\mathbb R}$ and all
$0 \leq j_{1},\hdots,j_{m}$. Then either
$\hat{P}(\lambda,y_{1},\hdots,y_{m})$ is convergent in a neighborhood of $y_{1}=\hdots=y_{m}=0$  or
$\hat{P}(\lambda) \not \in {\mathbb C}\{y_{1},\hdots,y_{m}\}$ for any $\lambda \in {\mathbb C}$
outside a polar set.
\end{pro}
  A polar set (see \cite{ran}) has measure zero and zero Haussdorff dimension.
Moreover, it is totally disconnected.
\begin{defi}
We define the derived equation
\[ \frac{\partial \alpha}{\partial x} =
\frac{\partial}{\partial \lambda} {\left({ \frac{-1}{(\log \varphi_{\lambda, \Delta})(x)} }\right)}_{|\lambda=0}
= \sum_{0 \leq j_{0},\hdots,j_{n}}
\frac{\partial{d_{j_{0},\hdots,j_{n}}}}{\partial \lambda}(0) x^{j_{0}} \prod_{k=1}^{n} x_{k}^{j_{k}}  \]
associated to ${\rm exp}(f \partial/\partial{x})$ and $\varphi_{\lambda, \Delta}$.
This equation is easier to handle than the homological equation since we can relate it to
\[ \frac{\partial \alpha}{\partial x} =
{\left({ \frac{\partial{(x \circ \varphi_{0})}}{\partial{x}} }\right)}^{-1}
{\left({ \frac{f}{(\log \varphi_{0})(x)} }\right)}^{2}
\Delta . \]
This equation will be called the reduced derived equation associated to
${\rm exp}(f \partial/\partial{x})$ and $\varphi_{\lambda, \Delta}$.
\end{defi}
The property ${\rm exp}(f \partial/\partial{x}) \stackrel{ut}{\sim}_{*} \varphi_{\lambda,\Delta}$
for any $\lambda \in {\mathbb C}$ implies convergence of special solutions
in $f_{N}=0$ for the homological equation associated to
${\rm exp}(f \partial/\partial{x})$ and $\varphi_{\lambda, \Delta}$
and any $\lambda \in {\mathbb C}$.
The dependence on $\lambda$ is holomorphic by proposition \ref{pro:PM}.
Thus there exists a solution of the derived equation that converges by restriction to $f_{N}=0$.
In order to prove that there exists $\lambda_{0} \in {\mathbb C}$ such that
${\rm exp}(f \partial/\partial{x}) \not \stackrel{ut}{\sim}_{*} \varphi_{\lambda_{0},\Delta}$
it suffices to show that the derived equation does not have solutions converging by
restriction to $f_{N}=0$. Indeed we can replace the derived equation with the simpler
reduced derived equation by the following proposition.
\begin{pro}
\label{pro:diffsim}
Fix $\varphi_{0} \in {\mathcal D}_{f}'$ and $\Delta \in \convn$.
We consider a union $E$ of some irreducible components of $f=0$.
Then the derived equation associated to ${\rm exp}(f \partial/\partial{x})$ and
$\varphi_{\lambda, \Delta}$ has a formal solution converging by restriction to $E$ if and only if
the reduced derived equation has a formal solution converging by restriction to $E$.
\end{pro}
The previous proposition implies proposition \ref{pro:linkder}.
\begin{proof}
Denote $\hat{u}_{\lambda}=[(\log \varphi_{\lambda,\Delta})(x)]/f$. Since
$x \circ \varphi_{\lambda,\Delta} - x \circ \varphi_{0} \in (f^{2})$ we can express
$\hat{u}_{\lambda}$ in the form
\[ \hat{u}_{\lambda} = \hat{u}_{0} + f \sum_{j=1}^{\infty} u_{j} \lambda^{j} \]
where $u_{j} \in \formn$ for any $j \in {\mathbb N}$. The derived equation is equal to
$\partial{\alpha}/\partial{x}=u_{1}/\hat{u}_{0}^{2}$. Now, we want to express
$(\partial{(x \circ \varphi_{\lambda,\Delta})}/\partial{\lambda})_{|\lambda=0}$
in terms of $u_{1}$. We have
\[ x \circ \varphi_{\lambda,\Delta} - \left({
 x + \sum_{j=1}^{\infty} \frac{(\log \varphi_{0})^{j}(x) + \lambda c_{j}}{j!} }\right)
\in (\lambda^{2}) \]
where $c_{1}=f^{2}u_{1}$ and
$c_{j+1} = (\log \varphi_{0})(c_{j}) + (fu_{1}/\hat{u}_{0}) (\log \varphi_{0})^{j+1}(x)$
for any $j \in {\mathbb N}$. From these formulas we can prove that
\[ c_{j} = \sum_{k=1}^{j}  {j \choose k-1} (\log \varphi_{0})^{k}(x) (\log \varphi_{0})^{j-k}
\left({ \frac{u_{1} f}{\hat{u}_{0}} }\right) \]
by induction. Remark that $(\log \varphi_{0})^{0}(u_{1}f / \hat{u}_{0}) = u_{1}f / \hat{u}_{0}$.
Therefore we obtain
\[ f^{2} \Delta = \sum_{k \geq 1, j \geq 0} \frac{1}{(j+k)!} {j+k \choose k-1}
(\log \varphi_{0})^{k}(x) (\log \varphi_{0})^{j}
\left({ \frac{u_{1} f}{\hat{u}_{0}} }\right) . \]
We simplify to get
\[ f^{2} \Delta = \sum_{k = 1}^{\infty} \frac{(\log \varphi_{0})^{k}(x)}{(k-1)!}
\sum_{j=0}^{\infty} \frac{(\log \varphi_{0})^{j}(u_{1}f/\hat{u}_{0})}{(j+1)!} . \]
Since $(\log \varphi_{0})(x \circ \varphi_{0}) = \sum_{k = 1}^{\infty} (\log \varphi_{0})^{k}(x)/(k-1)!$
then we have
\[ \sum_{j=0}^{\infty} \frac{(\log \varphi_{0})^{j}(u_{1}f/\hat{u}_{0})}{(j+1)!} =
{\left({ \frac{\partial{(x \circ \varphi_{0})}}{\partial{x}} }\right)}^{-1}
\frac{f}{\hat{u}_{0}} \Delta . \]
Now consider the equation
\[ \frac{\partial \alpha}{\partial{x}}=
\frac{1}{\hat{u}_{0} f}  \sum_{j=0}^{\infty} \frac{(\log \varphi_{0})^{j}(u_{1}f/\hat{u}_{0})}{(j+1)!}
= {\left({ \frac{\partial{(x \circ \varphi_{0})}}{\partial{x}} }\right)}^{-1}
\frac{\Delta}{\hat{u}_{0}^{2}} . \]
The last equation is the reduced derived equation; since the derived equation is
$\partial{\alpha}/\partial{x}=u_{1}/\hat{u}_{0}^{2}$ then it is enough to prove that
the equation
\[ \frac{\partial \alpha}{\partial{x}}=
\frac{1}{\hat{u}_{0} f}  \sum_{j=1}^{\infty} \frac{(\log \varphi_{0})^{j}(u_{1}f/\hat{u}_{0})}{(j+1)!}  \]
has a vanishing formal solution on $f=0$. That is clear since
\[ \hat{\alpha} =  \sum_{j=1}^{\infty} \frac{(\log \varphi_{0})^{j-1}(u_{1}f/\hat{u}_{0})}{(j+1)!} \]
is the desired solution.
\end{proof}
The reduced derived equation associated to
${\rm exp}(f \partial / \partial{x})$ and $\varphi_{\lambda,\Delta}$ is linear on $\Delta \in \convn$.
It is not convergent in general;
more precisely for $\Delta \not \equiv 0$ the right-hand side of the equation
belongs to $\convn$ if and only if $\log \varphi_{0}$ is convergent.
\section{Transport phenomenon}
\label{sec:tra}
In the next two sections we introduce the phenomena on $f$ producing the existence of
couples $\varphi_{1},\varphi_{2} \in {\mathcal D}_{f}$
such that $\varphi_{1} {\sim}_{*} \varphi_{2}$ but
$\varphi_{1} \not \stackrel{t}{\sim}_{*} \varphi_{2}$.
The first set of examples is contained in ${\mathcal D}_{f}$
for $f=x^{a} {(x-x_{1})}^{b}$. We relate the existence of u.t.f.
conjugations with properties of the reduced derived equation.

  Fix $f=x^{a} {(x-x_{1})}^{b} \in {\mathbb C}\{x,x_{1}\}$ for some
$(a,b) \in {\mathbb N}^{2}$. Throughout this section we denote $x_{1}$ by $y$.
Consider $\varphi \in {\mathcal D}_{f}' \subset \diff{up}{2}$.
Let $\hat{\alpha}_{\varphi} \in {\mathbb C}[[x,y]]$ be a solution of
the homological equation $E_{\varphi}$
associated to ${\rm exp}(f \partial / \partial{x})$ and $\varphi$.
Then the set of special solutions of
$E_{\varphi}$ is of the form $\hat{\alpha}_{\varphi} + {\mathbb C}[[y]]$. We define an operator
$S_{a,b}:{\mathcal D}_{f}' \to {\mathbb C}[[y]]$ given by
\[ S_{a,b}(\varphi) = \hat{\alpha}_{\varphi}(y,y) - \hat{\alpha}_{\varphi}(0,y) . \]
The definition of $S_{a,b}(\varphi)$ does not depend on the choice of $\hat{\alpha}_{\varphi}$.
\begin{pro}
\label{pro:lacoftra}
Fix $f=x^{a} {(x-y)}^{b}$ and $\varphi \in {\mathcal D}_{f}'$.
Then $S_{a,b}(\varphi)$ is convergent if and only if
${\rm exp}(f \partial/\partial{x}) \stackrel{t}{\sim}_{*} \varphi$.
\end{pro}
\begin{proof}
Implication $(\Leftarrow)$. There exists a solution $\hat{\alpha}_{\varphi} \in {\mathbb C}[[x,y]]$
of the homological equation associated to ${\rm exp}(f \partial/\partial{x})$ and $\varphi$
such that $\hat{\alpha}_{\varphi}$ converges by restriction to $x(x-y)=0$ by theorem \ref{teo:cartra}.
Then $S_{a,b}(\varphi)$ belongs to ${\mathbb C}\{y\}$.

  Implication $(\Rightarrow)$. Let $\hat{\alpha}_{\varphi} \in {\mathbb C}[[x,y]]$ be the solution
of the homological equation associated to ${\rm exp}(f \partial/\partial{x})$ and $\varphi$
such that $\hat{\alpha}_{\varphi}(0,y) \equiv 0$. This implies
$\hat{\alpha}_{\varphi}(y,y)=S_{a,b}(\varphi)$.
Thus $\hat{\alpha}_{\varphi}$ converges in $x(x-y)=0$. We are done by theorem \ref{teo:cartra}.
\end{proof}
Let $\hat{\alpha}_{\varphi} \in {\mathbb C}[[x,y]]$ be a solution
of the homological equation associated to ${\rm exp}(f \partial/\partial{x})$ and $\varphi$.
Once we choose $\hat{\alpha}_{\varphi}(0,y)$ the solution is ``transported'' and
$\hat{\alpha}_{\varphi}(y,y)$ is determined.
There always exists $\hat{\sigma}_{g,j} \in \diffh{p}{2}$ which is t.f. along $x=jy$ and conjugates
${\rm exp}(f \partial/\partial{x})$ and $\varphi$ for any $j \in \{0,1\}$.
But if $S_{a,b}(\varphi)$ diverges then
$\hat{\sigma}_{g,0}$ is not t.f. along $x=y$ whereas $\hat{\sigma}_{g,1}$ is not t.f. along $x=0$.

Fix $f=x^{a} {(x-y)}^{b}$ and $\varphi_{0} \in {\mathcal D}_{f}'$.
Given $\Delta \in {\mathbb C}\{x,y\}$ consider a solution
$\hat{\alpha}_{\Delta} \in {\mathbb C}[[x,y]]$ of
the derived equation associated to ${\rm exp}(f \partial/\partial{x})$ and
$\varphi_{\lambda,\Delta}$. We
define the operator $S_{a,b,\varphi_{0}}:{\mathbb C}\{x,y\} \to {\mathbb C}[[y]]$ given by
\[ S_{a,b,\varphi_{0}}(\Delta)(y) = \hat{\alpha}_{\Delta}(y,y)- \hat{\alpha}_{\Delta}(0,y) . \]
The operator is well-defined. Moreover by proposition \ref{pro:diffsim}
we can replace in the definition of $S_{a,b,\varphi_{0}}(\Delta)$ the series
$\hat{\alpha}_{\Delta}$ by a formal solution of the reduced derived equation associated to
${\rm exp}(f \partial/\partial{x})$ and $\varphi_{\lambda,\Delta}$.
\begin{pro}
\label{pro:ptf1}
Fix $f=x^{a} {(x-y)}^{b}$ and $\varphi_{0} \in {\mathcal D}_{f}'$. We have
\[ \frac{\partial{S_{a,b}(\varphi_{\lambda,\Delta})}}{\partial{\lambda}}_{|\lambda=0} = S_{a,b,\varphi_{0}}(\Delta) \]
for any $\Delta \in {\mathbb C}\{x,y\}$.
Moreover if $S_{a,b,\varphi_{0}}({\mathbb C}\{x,y\}) \not \subset {\mathbb C}\{y\}$ then
there exists $\varphi \in {\mathcal D}_{f}'$ such that
${\rm exp}(f \partial/\partial{x}) {\sim}_{*} \varphi$
but ${\rm exp}(f \partial/\partial{x}) \not \stackrel{t}{\sim}_{*} \varphi$.
\end{pro}
\begin{proof}
The first part is true by definition. The proposition \ref{pro:polbnd} implies that
$S_{a,b}(\varphi_{\lambda,\Delta})$
is of the form $\sum_{j=1}^{\infty} \beta_{j}(\lambda) {y}^{j}$
where $\beta_{j} \in {\mathbb C}[\lambda]$ satisfies $\deg \beta_{j} \leq j+(a+b-1)$
for any $j \in {\mathbb N}$.
Suppose we have $\Delta \in {\mathbb C}\{x,y\}$ such that
$S_{a,b,\varphi_{0}}(\Delta) \not \in {\mathbb C}\{y\}$.
Thus we have that $S_{a,b}(\varphi_{\lambda,\Delta}) \not \in {\mathbb C}\{y\}$ for any
$\lambda \not \in E$ where $E$ is a polar set (prop. \ref{pro:PM}). Now choose
$\lambda_{0} \not \in E$ and define
$\varphi=\varphi_{\lambda_{0},\Delta}$. On the one hand
$\varphi  {\sim}_{*} {\rm exp}(f \partial/\partial{x})$
since $\varphi \in {\mathcal D}_{f}'$. On the other hand we have
$\varphi \not \stackrel{t}{\sim}_{*} {\rm exp}(f \partial/\partial{x})$
by proposition \ref{pro:lacoftra}.
\end{proof}
\section{Bad position with respect to $\partial/\partial{x}$}
\label{sec:bad}
Fix $f = {(x_{2}-xx_{1})}^{c} \in {\mathbb C}\{x,x_{1},x_{2}\}$
for some $c \in {\mathbb N}$. Throughout this section we denote
$x_{1}$ by $y$ and $x_{2}$ by $z$. Consider $\varphi \in {\mathcal D}_{f}'$.
Let $\hat{\alpha}_{\varphi} \in {\mathbb C}[[x,y,z]]$ be a solution of the
homological equation $E_{\varphi}$
associated to ${\rm exp}(f \partial{x})$ and $\varphi$. Then the set of special solutions of
$E_{\varphi}$ is of the form $\hat{\alpha}_{\varphi} + {\mathbb C}[[y,z]]$. We can express
$\hat{\alpha}_{\varphi}(x,y,xy)$ in the form
$\sum_{0 \leq j,k} \alpha_{j,k}(\varphi) {x}^{j} {y}^{k}$. We define an operator
$T_{c}:{\mathcal D}_{f}' \to {\mathbb C}[[x,y]]$ given by
\[ T_{c}(\varphi) = \sum_{0 \leq k < j} \alpha_{j,k}(\varphi) {x}^{j} {y}^{k} . \]
The definition of $T_{c}(\varphi)$ does not depend on the choice of $\hat{\alpha}_{\varphi}$.
\begin{pro}
Fix $f={(z-xy)}^{c}$ and $\varphi \in {\mathcal D}_{f}'$. Then
$T_{c}(\varphi)$ is convergent if and only if
${\rm exp}(f \partial/\partial{x}) \stackrel{t}{\sim}_{*} \varphi$.
\end{pro}
\begin{proof}
Implication $(\Leftarrow)$. By theorem \ref{teo:cartra} there exists a special solution
$\hat{\alpha}_{\varphi}$ of $\partial{\alpha}/\partial{x} =1/f - 1/(\log \varphi)(x)$
converging by restriction to $f=0$. Thus $\hat{\alpha}_{\varphi}(x,y,xy)$ belongs
to ${\mathbb C}\{x,y\}$ and then $T_{c}(\varphi) \in {\mathbb C}\{x,y\}$.

Implication $(\Rightarrow)$. Let $\hat{\alpha}_{\varphi}$ be a special solution
of the homological equation associated to ${\rm exp}(f \partial/\partial{x})$ and $\varphi$.
We define
\[ \hat{\beta} = \hat{\alpha}_{\varphi} - \sum_{0 \leq j \leq k} \alpha_{j,k}(\varphi) y^{k-j} z^{j}. \]
Then $\hat{\beta}$ is  a solution of the homological equation. Moreover, since
$\hat{\beta}(x,y,xy)=T_{c}(\varphi)(x,y)$ then $\hat{\beta}$ converges by restriction to $f=0$.
We obtain ${\rm exp}(f \partial/\partial{x}) \stackrel{t}{\sim}_{*} \varphi$ by theorem \ref{teo:cartra}.
\end{proof}
Fix $f={(z-xy)}^{c}$ and $\varphi_{0} \in {\mathcal D}_{f}'$.
Given $\Delta \in {\mathbb C}\{x,y,z\}$ consider a formal solution
$\hat{\alpha}_{\Delta}=\sum_{j,k,l} \alpha_{j,k,l} x^{j} y^{k} z^{l}$ of
the derived equation associated to ${\rm exp}(f \partial/\partial{x})$
and $\varphi_{\lambda,\Delta}$.
Let $T_{c,\varphi_{0}}:{\mathbb C}\{x,y,z\} \to {\mathbb C}[[x,y]]$ be the operator given by
\[ T_{c,\varphi_{0}}(\Delta)(x,y) = \sum_{k < j} \alpha_{j,k,l} x^{j+l} y^{k+l} . \]
The operator is well-defined.
Again we can replace the derived equation by the reduced derived equation in the definition.
The next proposition is proved in an analogous way than proposition \ref{pro:ptf1}.
\begin{pro}
\label{pro:ptf2}
Fix $f={(z-xy)}^{c}$ and $\varphi_{0} \in {\mathcal D}_{f}'$. We have
\[ \frac{\partial{T_{c}(\varphi_{\lambda,\Delta})}}{\partial{\lambda}}_{|\lambda=0} =
T_{c,\varphi_{0}}(\Delta)  \]
for any $\Delta \in {\mathbb C}\{x,y,z\}$.
Moreover if $T_{c,\varphi_{0}}({\mathbb C}\{x,y,z\}) \not \subset {\mathbb C}\{x,y\}$ then
there exists $\varphi \in {\mathcal D}_{f}'$ such that
${\rm exp}(f \partial/\partial{x}) {\sim}_{*} \varphi$
but ${\rm exp}(f \partial/\partial{x}) \not \stackrel{t}{\sim}_{*} \varphi$.
\end{pro}
Here there is no transport phenomenon since
$\sharp \{ z_{0}-xy_{0}=0 \} \leq 1$ for $(y_{0},z_{0}) \neq (0,0)$.
We already know that the nature of fibered and non-fibered irreducible components of
$f=0$ is different. This case is hybrid since $z=xy$ is non-fibered but
contains the fibered line $y=z=0$. The lack of transversality of $z-xy=0$
with respect to $\partial / \partial x$ is to blame for the lack of t.f.
conjugations.

We can naturally increase the transversality between $z=xy$ and
$\partial/\partial{x}$ by considering
the blow-up $\pi:\widetilde{{\mathbb C}^{3}} \to {\mathbb C}^{3}$ with center
at the line $y=z=0$. Since the tangent cone of $z-xy=0$ is $z=0$ we consider
the chart $x=x$, $y=s$, $z=st$.
We obtain
\[ (f \circ \pi)(x,s,t)= {s}^{c} {(t-x)}^{c} . \]
Denote the point $(x,s,t)=(0,0,0)$ by $q$. The divisor
$s=0$ is a fibered irreducible component of $f \circ \pi=0$ whereas the strict transform $x=t$
of $f=0$ is transversal to
$\partial/\partial{x}$ at $q$. Since we can find a solution of the homological equation vanishing
on $t=x$ then there exists a t.f. $\hat{\sigma}_{\varphi}$ in
$\widehat{\rm Diff}_{p}(\widetilde{{\mathbb C}^{3}},q)$
conjugating ${\rm exp}(f \partial/\partial{x})$ and $\varphi$ for any
$\varphi \in {\mathcal D}_{f}'$. In spite of this there is no choice
in general of $\hat{\sigma}_{\varphi}$
such that $\pi \circ \hat{\sigma}_{\varphi} \circ \pi^{(-1)}$ extends to an element of $\diffh{p}{3}$.
The way $z-xy=0$ folds around $y=z=0$ is avoiding the existence of t.f. conjugations.
\section{Proof of the main theorem}
In this section we prove that the equivalence relations
${\sim}_{*}$ and $\stackrel{t}{\sim}_{*}$
do not define the same classes of equivalence. Some technical details
are postponed for the next sections.

Fix $v \in {\mathbb C}[[x,y]]$. We define
$L_{2}^{v}:{\mathbb C}[[x,y]] \to {\mathbb C}[[y]]$
given by
\[ L_{2}^{v}(g) = \hat{\alpha}_{g}(y,y) - \hat{\alpha}_{g}(0,y) \]
where $\hat{\alpha}_{g} \in {\mathbb C}[[x,y]]$ is a solution
of $\partial{\alpha}/\partial{x}=v g$.

Fix $v \in {\mathbb C}[[x,y,z]]$. We define
$L_{3}^{v}:{\mathbb C}[[x,y,z]] \to {\mathbb C}[[x,y]]$
given by
\[ L_{3}^{v}(g) = \sum_{k < j} \alpha_{j,k,l} x^{j+l} y^{k+l}  \]
where $\hat{\alpha}_{g}=\sum_{j,k,l}   \alpha_{j,k,l} x^{j} y^{k} z^{l} \in {\mathbb C}[[x,y,z]]$
is a solution of $\partial{\alpha}/\partial{x}=v g$.

The next propositions and prop. \ref{pro:exlodi}
will be proved in next sections.
\begin{pro}
\label{pro:mosiop2}
Fix $v \in {\mathbb C}[[x,y]]$.
Then $L_{2}^{v}({\mathbb C}\{x,y\}) \subset {\mathbb C}\{y\}$ implies $v \in {\mathbb C}\{x,y\}$.
\end{pro}
\begin{pro}
\label{pro:mosiop3}
Fix $v \in {\mathbb C}[[x,y,z]]$.
Then $L_{3}^{v}({\mathbb C}\{x,y,z\}) \subset {\mathbb C}\{x,y\}$ implies $v \in {\mathbb C}\{x,y,z\}$.
\end{pro}
The following theorems provide the examples for the Main Theorem.
\begin{teo}
\label{teo:latf2}
Fix $f=x^{a} {(x-y)}^{b}$. There exists $\varphi \in {\mathcal D}_{f}' \subset \diff{up}{2}$
such that ${\rm exp}(f \partial/\partial{x}) {\sim}_{*} \varphi$
but ${\rm exp}(f \partial/\partial{x}) \not \stackrel{t}{\sim}_{*} \varphi$.
\end{teo}
\begin{teo}
\label{teo:latf3}
Fix $f={(z-xy)}^{c}$. There exists $\varphi \in {\mathcal D}_{f}' \subset \diff{up}{3}$ such that
${\rm exp}(f \partial/\partial{x}) {\sim}_{*} \varphi$
but ${\rm exp}(f \partial/\partial{x}) \not \stackrel{t}{\sim}_{*} \varphi$.
\end{teo}
Obviously theorems \ref{teo:latf2} and \ref{teo:latf3}
imply theorems \ref{teo:intr2} and \ref{teo:intr3} respectively.
\begin{proof}[proof of proposition \ref{pro:derlog}]
We have
\[ S_{a,b,\varphi_{0}}({\mathbb C}\{x,y\}) \subset {\mathbb C}\{y\} \implies
L_{2}^{v}({\mathbb C}\{x,y\}) \subset {\mathbb C}\{y\} \]
where $v = {(\partial{(x \circ \varphi_{0})}/\partial{x})}^{-1}
{( f / (\log \varphi_{0})(x) )}^{2}$ by proposition \ref{pro:diffsim}.
Thus $\log  \varphi_{0}$ converges by proposition \ref{pro:mosiop2}.
\end{proof}
\begin{proof}[proof of theorems \ref{teo:latf2} and \ref{teo:latf3}]
We suppose that we are in the situation described in theorem \ref{teo:latf2}. Otherwise
the proof is analogous. Consider $\varphi_{0} \in {\mathcal D}_{f}'$ such that
$\log \varphi_{0}$ is divergent; it is possible by proposition \ref{pro:exlodi}
(the proof is in section \ref{sec:div}).
We obtain
$S_{a,b,\varphi_{0}}({\mathbb C}\{x,y\}) \not \subset {\mathbb C}\{y\}$
by prop. \ref{pro:derlog}. The result is a consequence of prop.  \ref{pro:ptf1}.
\end{proof}
\begin{proof}[proof of the Main Theorem]
Fix $f=x^{a} {(x-x_{1})}^{b}$ for some $(a,b) \in {\mathbb N}^{2}$.
Let $\varphi_{1} = {\rm exp}(f \partial/\partial{x})$.
There exists $\varphi_{2} \in {\mathcal D}_{f}' \subset \diff{up}{2}$
such that $\varphi_{1} \not \stackrel{t}{\sim}_{*} \varphi_{2}$ by theorem \ref{teo:latf2}.
We claim that there is no $\hat{\sigma} \in \diffh{}{2}$ such that it is t.f.
along $x(x-x_{1})=0$ and conjugates $\varphi_{1}$ and $\varphi_{2}$.
Suppose this is false. The series $x_{1} \circ \hat{\sigma} \in {\mathbb C}[[x_{1}]]$
is t.f. along $x=0$ and then $x_{1} \circ \hat{\sigma} \in {\mathbb C}\{ x_{1} \}$.
Let us consider $\sigma \in \diff{}{2}$ such that $x_{1} \circ \sigma = x_{1} \circ \hat{\sigma}$
and $x \circ \hat{\sigma} - x \circ \sigma \in (f^{2})$.
The mapping $\sigma \circ \varphi_{1} \circ \sigma^{(-1)}$
belongs to ${\mathcal D}_{f}'$ by the choice of $\sigma$. Since
\[ \log(\sigma \circ \varphi_{1} \circ \sigma^{(-1)}) =
\frac{f \circ \sigma^{(-1)}}{\partial{(x \circ \sigma^{(-1)})}/\partial{x}} \frac{\partial}{\partial{x}} \]
is convergent then
${\rm exp}(f \partial/\partial{x})$ and $\sigma \circ \varphi_{1} \circ \sigma^{(-1)}$
are conjugated by a normalized $\eta \in \diff{p}{2}$. Therefore
$\hat{\sigma} \circ \sigma^{(-1)} \circ \eta$ is a normalized t.f. element of $\diffh{p}{2}$
conjugating $\varphi_{1}$ and $\varphi_{2}$. That is contradictory with our choice of $\varphi_{1}$
and $\varphi_{2}$.

For $n \geq 1$ and $j \in \{1,2\}$ we define
\[ \varphi_{j,n}(x,x_{1},\hdots,x_{n}) = (x \circ \varphi_{j}(x,x_{1}),x_{1}, \hdots,x_{n}) . \]
Fix $n \in {\mathbb N}$.
Since $\varphi_{1,n}, \varphi_{2,n} \in {\mathcal D}_{f}' \subset \diff{up}{n+1}$
then $\varphi_{1,n} {\sim}_{*} \varphi_{2,n}$ by theorem \ref{teo:UPD}.
We claim that there does not exist $\hat{\sigma}_{n} \in \diffh{}{n+1}$ such that it is t.f.
along $x(x-x_{1})=0$ and conjugates $\varphi_{1,n}$ and $\varphi_{2,n}$.
Suppose it is false;
the property $\hat{\sigma}_{n} \{ x(x-x_{1})=0 \} = \{ x(x-x_{1})=0 \}$ implies that the first jet of
\[ \hat{\xi}_{n}(x,x_{1}) =
(x \circ \hat{\sigma}_{n}(x,x_{1},0,\hdots,0), x_{1} \circ \hat{\sigma}_{n}(x,x_{1},0,\hdots,0)) \]
is an invertible linear mapping and then $\hat{\xi}_{n} \in \diffh{}{2}$.
Clearly $\hat{\xi}_{n}$ is t.f. along $x(x-x_{1})=0$ and it satisfies
$\hat{\xi}_{n} \circ \varphi_{1} = \varphi_{2} \circ \hat{\xi}_{n}$.
That is contradictory with the first part of the proof.
\end{proof}
\begin{rem}
Let $f={(x_{2}-xx_{1})}^{c}$.
We can choose the examples $\varphi_{1},\varphi_{2}$ provided by the main theorem
in ${\mathcal D}_{f}' \subset \diff{up}{n+1}$ for any $n \geq 2$.
The proof is analogous to the previous one.
\end{rem}
\begin{rem}
Consider $f=x^{a} {(x-x_{1})}^{b}$.
Suppose that there exists $\hat{\sigma} \in \diffh{}{n+1}$
conjugating ${\rm exp}(f \partial / \partial x)$ and $\varphi \in {\mathcal D}'(f)$
such that $\hat{\sigma}$ is t.f. exactly along one irreducible component of $f=0$.
By changing slightly the previous proof we can show that
${\rm exp}(f \partial / \partial x)$ and $\varphi \in {\mathcal D}'(f)$ are not analytically
conjugated. Intuitively, if a homological equation has a solution that converges
in exactly one component then $S_{a,b}(\varphi)$ is divergent.
This is an obstruction to the existence of analytic conjugations.
\end{rem}
\section{Divergence of the infinitesimal generator}
\label{sec:div}
\begin{proof}[proof of prop. \ref{pro:exlodi}]
 Suppose $f_{N}(0) \neq 0$. Consider $\tau \in {\mathcal D}_{x_{1}} \subset \diff{}{2}$
such that $\log \tau$ is divergent (Voronin's paper in \cite{rus}).
Moreover $\log \tau$ is of the form $\hat{u}(x,x_{1}) x_{1} \partial/\partial{x}$
where $\hat{u}$ diverges.
Consider the mapping $\sigma:{\mathbb C}^{n+1} \to {\mathbb C}^{2}$ given by
$\sigma(x,x_{1},\hdots,x_{n})=(x,f_{F}(x_{1},\hdots,x_{n}))$.
Now $\varphi_{0} = \sigma^{(-1)} \circ \tau \circ \sigma$ is an element of $\diff{up}{n+1}$.
Moreover $\log \varphi_{0}$ is equal to $(\hat{u} \circ \sigma) f_{F} \partial/\partial{x}$.
As a consequence $\varphi_{0}$ is an element of ${\mathcal D}_{f}$ whose
infinitesimal generator is divergent.
Since $\hat{u} \circ \sigma$ is t.f. along $f=0$ there exists $u_{0} \in \convn$ such
that $(\hat{u} \circ \sigma)/f_{N} - u_{0} \in (f)$. The homological equation associated to
${\rm exp}(u_{0}f \partial{x})$ and ${\rm exp}(f \partial/\partial{x})$ is special.
Thus there exists $\xi \in \diff{p}{n+1}$ such that
\[ \xi \circ {\rm exp}(u_{0}f \partial / \partial{x}) = {\rm exp}(f \partial / \partial{x}) \circ \xi. \]
We deduce that $\xi \circ \varphi_{0} \circ \xi^{(-1)}$ is an element of ${\mathcal D}_{f}'$
whose infinitesimal generator is divergent.

  Consider the decomposition $\prod_{j=1}^{p} f_{j}^{l_{j}}$ of $f_{N}$ in irreducible factors.
Suppose there exists $1 \leq k \leq p$ such that $l_{k} \geq 2$. Consider an open neighborhood
$U$ of the origin such that $f \in {\mathcal O}(U)$. We can suppose that
$S= (\{ f_{k}=0 \} \cap \{ \partial{f_{k}}/\partial{x} \neq 0\}) \setminus \cup_{j \neq k} \{f_{j}=0\}$
is connected in $U$. We choose a point $q =(x^{0},x_{1}^{0},\hdots,x_{n}^{0}) \in U \cap S$.
By the choice of $q$ the function $f_{k}(x,x_{1}^{0},\hdots,x_{n}^{0})$ is a coordinate in the line
$\cap_{j=1}^{n} (x_{j}=x_{j}^{0})$
in the neighborhood of $x=x^{0}$. The one-variable theory of tangent to the identity diffeomorphisms
implies the existence of $\sum_{j=0}^{\infty} \lambda_{j} z^{j} \in {\mathbb C}\{z\}$ such that
the infinitesimal generator of
\[ \varphi_{0} = (x \circ {\rm exp}(f \partial/\partial{x}) +
f^{2} \sum_{j=0}^{\infty}\lambda_{j}f_{k}^{j},x_{1},\hdots,x_{n}) \]
restricted to $\cap_{j=1}^{n} (x_{j}=x_{j}^{0})$ does not belong to
${\mathcal X} ({\mathbb C},x^{0})$. We define
\[ T = \{ (x^{1},x_{1}^{1},\hdots,x_{n}^{1}) \in S \cap U :
(\log \varphi_{0})_{|\cap_{j=1}^{n} (x_{j}=x_{j}^{1})} \not \in {\mathcal X} ({\mathbb C},x^{1}) \} . \]
By the one-variable theory $S \setminus T$ is analytic in $S$. Since $q \in T$ then the origin is
contained in the closure of $T$. Therefore $\log \varphi_{0}$ diverges.

  Suppose that $l_{j}=1$ for any $1 \leq j \leq p$. Choose any $1 \leq k \leq p$.
There exists a sequence of points
$q_{r}=(x^{r},x_{1}^{r},\hdots,x_{n}^{r}) \in \{f_{k}=0\}$ ($r \in {\mathbb N}$) such that
$\lim_{r \to \infty} q_{r} = (0,\hdots,0)$ and
$(\partial{(x \circ {\rm exp}(f \partial{x}))}/\partial{x})(q_{r})$
is a $c_{r}$-root of the unit for some $c_{r} \geq 2$. We can suppose that the
sequence $c_{r}$ is strictly
increasing. Let $\lambda = (\lambda_{s})_{s \in {\mathbb N}}$ be a sequence
of complex numbers. We define
\[ \eta_{\lambda} = (x \circ {\rm exp}(f \partial/\partial{x}) + \sum_{j=2}^{\infty} \lambda_{j} f^{c_{j}+1},
x_{1},\hdots,x_{n}) . \]
Given $\lambda_{1}, \hdots, \lambda_{r-1}$ there exists
$0< K(\lambda_{1},\hdots,\lambda_{r-1}) < 1/(c_{r}+1)!$ such that
$(\eta_{\lambda})^{(c_{r})}_{\cap_{j=1}^{n} (x_{j}=x_{j}^{r})}$ is not the identity in the
neighborhood of $x=x^{r}$ if
$0<|\lambda_{r}|<K(\lambda_{1},\hdots,\lambda_{r-1})$ and $\lambda_{s} \in {\mathbb C}$ for
any $s > r$. Then $(\eta_{\lambda})_{\cap_{j=1}^{n} (x_{j}=x_{j}^{r})}$ is not the exponential of
an element of $\hat{\mathcal X} ({\mathbb C},x^{r})$ since then it would be periodic.
As a consequence we can obtain by induction a sequence
$0 < |\lambda_{j}| < 1/(c_{j}+1)!$ for any $j \geq 2$
such that $\log \eta_{\lambda}$ is divergent. We choose $\varphi_{0}=\eta_{\lambda}$.
\end{proof}
\section{The operator $L_{2}^{v}$}
\label{section:twodim}
The goal of this section is proving proposition \ref{pro:mosiop2}.
The techniques were already used in \cite{nonormal}.

Let  $B^{2} \subset {\mathbb C}[[x,y]]$ be the Banach space whose elements are
the power series  $H=\sum_{0 \leq j,k} H_{j,k} {x}^{j} {y}^{k}$ such that
\[ ||H||  = \sum_{0 \leq j,k} |H_{j,l}| < +\infty . \]
We have $B^{2} \subset {\mathcal O}(B(0,1)^{2})$. Moreover, a function $H \in B^{2}$
is continuous in $\overline{B(0,1)} \times \overline{B(0,1)}$. Given $v$ in  ${\mathbb C}[[x,y]]$
we can define for $j \geq 1$ the linear functionals $L_{2,j}^{v}:B^{2} \to {\mathbb C}$
such that
\[ L_{2}^{v}(H) = \sum_{j \geq 1} L_{2,j}^{v}(H) {y}^{j} \]
for any $H \in B^{2}$.
\begin{lem}
Let $v \in {\mathbb C}[[x,y]]$. Then $L_{2,j}^{v}$ is a linear continuous functional
for any $j \in {\mathbb N}$.
\end{lem}
\begin{proof}
We denote $H=\sum_{0 \leq k,l} H_{k,l}(H) {x}^{k} {y}^{l}$. We have that
\[ L_{2,j}^{v} = \sum_{k+l<j} c_{k,l}^{j} H_{k,l} \]
where $c_{k,l}^{j} \in {\mathbb C}$ for all $j \geq 1$ and $k+l < j$. As a consequence
we obtain
$||L_{2,j}^{v}|| \leq \max_{k+l <j} |c_{k,l}^{j}|$.
\end{proof}
\begin{lem}
\label{lem:boupri}
Let $v \in {\mathbb C}[[x,y]]$. Either
$\lim \sup_{j \to \infty} \sqrt[j]{||L_{2,j}^{v}||} < + \infty$ or
$L_{2}^{v}(H) \not \in {\mathbb C}\{y\}$ for any $H$ in a dense subset of $B^{2}$.
\end{lem}
\begin{proof}
  Suppose $\lim \sup_{j \to \infty} \sqrt[j]{||L_{2,j}^{v}||} = + \infty$. We choose a sequence
$(a_{j})$ of positive numbers such that $a_{j} \to \infty$ and
\[ \lim \sup_{j \to \infty} \frac{\sqrt[j]{||L_{2,j}^{v}||}}{a_{j}} = + \infty. \]
Hence $\lim \sup_{j \to \infty} ||L_{2,j}^{v} /a_{j}^{j}|| = + \infty$. We deduce that
\[ \lim \sup_{j \to \infty} |L_{2,j}^{v}(H)| / a_{j}^{j} = + \infty \]
for any $H$ in a dense subset $E$ of $B^{2}$ by the uniform boundedness principle.
Moreover, since
\[ \lim \sup_{j \to \infty} \sqrt[j]{|L_{2,j}^{v}(H)|} \geq \lim \inf_{j \to \infty} a_{j} = + \infty \]
then $L_{2}^{v}(H) \not \in {\mathbb C}\{ y \}$ for any $H \in E$.
\end{proof}
\begin{pro}
Let $v \in {\mathbb C}[[x,y]]$. Suppose $L_{2}^{v}(B^{2}) \subset {\mathbb C}\{y\}$.
Then there exists $\eta>0$ such that
$L_{2}^{v}(H) \in {\mathcal O}(B(0,\eta ))$ for any $H \in B^{2}$.
\end{pro}
\begin{proof}
   There exists $\eta >0$ such that
$\lim \sup_{j \to \infty} \sqrt[j]{||L_{2,j}^{v}||} \leq 1 / \eta $
by lemma \ref{lem:boupri}. As a consequence
\[ \lim \sup_{j \to \infty} \sqrt[j]{|L_{2,j}^{v}(H)|} \leq
\lim \sup_{j \to \infty} \left({ \sqrt[j]{||L_{2,j}^{v}||} \sqrt[j]{||H|| } }\right)
\leq  1 / \eta  . \]
That implies that $L_{2}^{v}(H) \in {\mathcal O}(B(0,\eta ))$ for any $H \in B^{2}$.
\end{proof}
\begin{proof}[Proof of prop. \ref{pro:mosiop2}]
Since $L_{2}^{v}(B^{2}) \subset {\mathbb C} \{ y \}$ then there exists $C \geq 1$ such that
$||L_{2,j}^{v}|| \leq C^{j}$ for any $j \geq 1$ by lemma \ref{lem:boupri}.
We denote $v = \sum_{0 \leq k,l} v_{k,l} {x}^{k} {y}^{l}$. We have
\[ L_{2,1}^{v}(1)=v_{0,0} \Rightarrow |v_{0,0}| \leq ||L_{2,1}^{v}|| ||1||  \leq C. \]
Analogously we want to estimate $v_{k,0}$, $\hdots$, $v_{0,k}$ for any $k \geq 0$. We obtain
\[ \mbox{Hilb}^{k}
\left({
\begin{array}{c}
v_{0,k} \\
v_{1,k-1} \\
\vdots \\
v_{k,0}
\end{array}
}\right)
=
\left({
\begin{array}{c}
L_{2,k+1}^{v}(1) \\
L_{2,k+2}^{v}(x) \\
\vdots \\
L_{2,2k+1}^{v}(x^{k})
\end{array}
}\right)  \]
where $\mbox{Hilb}^{k}$ is the $(k+1) \times (k+1)$ Hilbert matrix; this is a real symmetric matrix such that
$\mbox{Hilb}_{a,b}^{k} = 1 /(a+b-1)$ for $1 \leq a,b \leq k+1$.
Moreover $\mbox{Hilb}^{k}$ is positive definite and following \cite{Kalyabin} we obtain that
\[ ||{(\mbox{Hilb}^{k})}^{-1}||_{2} = \frac{{\rho}^{4k}}{K \sqrt{k}} (1+o(1)) \]
where ${||\hdots||}_{2}$ is the spectral norm, $K=(8 \pi^{3/2} 2^{3/4})/{(1+ \sqrt{2})}^{4}$
and $\rho =1 + \sqrt{2}$.
We have $|L_{2,k+l+1}^{v}(x^{l})| \leq ||L_{2,k+l+1}^{v}|| ||{x}^{l}|| \leq C^{k+l+1}$.
As a consequence we obtain
\[ ||v_{0,k}, \hdots, v_{k,0}||_{2} \leq \frac{{\rho}^{4k}}{K \sqrt{k}} \sqrt{k+1} C^{2k+1} (1+h(k)) . \]
where $\lim_{k \to \infty} h(k)=0$. This implies that
\[ |v_{l,m}| \leq \frac{{\rho}^{4(l+m)}}{K \sqrt{l+m}} \sqrt{l+m+1} C^{2(l+m)+1} (1+h(l+m)) \]
for $0 \leq l,m$ and then
$v \in {\mathcal O}(B(0,{\rho}^{-4} {C}^{-2})^{2})$.
\end{proof}
\section{The operator $L_{3}^{v}$}
Analogously we define
the Banach space $B^{3} \subset {\mathbb C}[[x,y,z]]$ whose elements
$H=\sum_{0 \leq j,k,l} H_{j,k,l} {x}^{j} {y}^{k} {z}^{l}$ satisfy
$||H||  = \sum_{0 \leq j,k,l} |H_{j,k,l}| < +\infty$.
We can define the operators $L_{3,j,k}^{v}:B^{3} \to {\mathbb C}$ for
$0 \leq k < j$ such that
$L_{3}^{v}(H) = \sum_{0 \leq k < j} L_{3,j,k}^{v}(H) x^{j} y^{k}$.
The following lemmas are analogous to those in section \ref{section:twodim}.
\begin{lem}
Let $v \in {\mathbb C}[[x,y,z]]$. Then $L_{3,j,k}^{v}$ is a linear continuous functional
for all $0 \leq k < j$.
\end{lem}
\begin{lem}
\label{lem:boupri2}
Either $\sup_{0 \leq k < j} \sqrt[j+k]{||L_{3,j,k}^{v}||} < + \infty$ or
$L_{3}^{v}(H) \not \in {\mathbb C}\{x,y\}$ for any $H$ in a dense subset of $B^{3}$.
\end{lem}
\begin{proof}[proof of prop. \ref{pro:mosiop3}]
Since $L_{3}^{v}(B^{3}) \subset {\mathbb C} \{ x,y \}$ then there exists $C \geq 1$ such that
$||L_{3,j,k}^{v}|| \leq C^{j+k}$ for all $0 \leq k < j$ by lemma \ref{lem:boupri2}.
We denote $v = \sum_{0 \leq j,k,l} v_{j,k,l} {x}^{j} {y}^{k} {z}^{l}$.
Fix $j,k,l \in {\mathbb N}$. Denote $a=j+l$, $b=k+l$ and $d=\max(b-a,0)$. We have
\[ \mbox{Hilb}^{a+d}
\left({
\begin{array}{c}
v_{-d,b-a-d,a+d} \\
v_{-d+1,b-a-d+1,a+d-1} \\
\vdots \\
v_{a,b,0}
\end{array}
}\right)
=
\left({
\begin{array}{c}
L_{3,a+d+1,b}^{v}(x^{d}) \\
L_{3,a+d+2,b}^{v}(x^{d+1}) \\
\vdots \\
L_{3,2a+2d+1,b}^{v}(x^{a+2d})
\end{array}
}\right)  . \]
The terms $v_{\alpha,\beta,\gamma}$ where any subindex is negative are zero by definition.
Proceeding like in the proof of proposition \ref{pro:mosiop2} we can show
\[ |v_{j,k,l}| \leq C \frac{{(\rho^{4} C^{3})}^{j+k+l}}{K} (1+h(j,k,l))  \]
where $\lim_{j+k+l \to \infty} h(j,k,l)=0$. Thus
$v \in {\mathcal O}(B(0,\rho^{-4} C^{-3
})^{3})$.
\end{proof}
\bibliography{rendu}
\end{document}